\documentclass[11pt, a4paper]{article}

\title{$L^{2}$-estimates for the linear elastic waves}
\author{
Hiroshi Takeda\thanks{h-takeda@fukuoka-u.ac.jp},\\
Department of Applied Mathematics, \\
Faculty of Science, Fukuoka University, \\
Nanakuma, Jonan-ku, Fukuoka, 814-0180, Japan
}
\date{}

\usepackage{amssymb,amsmath,amsthm}
\usepackage[dvips]{graphicx}
\usepackage{bm} 
\usepackage{amsfonts}

\newcommand{\R}{\mathbb R}

\newcommand{\supp}{\mathop{\mathrm{supp}}\nolimits}

\newtheorem{thm}{Theorem}[section]
\newtheorem{cor}[thm]{Corollary}
\newtheorem{prop}[thm]{Proposition}
\newtheorem{lem}[thm]{Lemma}
\theoremstyle{remark}
\newtheorem{rem}[thm]{Remark}

\theoremstyle{definition}

\title{$L^{2}$-estimates for the linear elastic waves}
\date{}
\begin{document}
\maketitle

\numberwithin{equation}{section}
\begin{abstract}
This paper is concerned with the large time behavior of the solution to the Cauchy problem for the elastic wave equations. 
In particular, optimal $L^{2}$ estimates of the elastic waves are obtained
in the sense that the upper and lower bounds of the $L^{2}$ norms of each component of the solution are proved for large $t$, 
under the minimum assumptions necessary regarding regularity with respect to initial data. 
The proof is based on the approximation of the solution by a smooth auxiliary function with suitable parameters.
\end{abstract}

\section{Introduction}
In this paper, we consider the Cauchy problem of the elastic wave equation: 
\begin{equation} \label{eq:1.1}
\left\{
\begin{split}
& \partial_{t}^{2} u -\mu \Delta u -(\lambda+\mu) \nabla ({\rm div} u) =0, \quad t>0, \quad x \in \R^{n}, \\
& u(0,x)=f_{0}(x), \quad \partial_{t} u(0,x)=f_{1}(x) , \quad x \in \R^{n} 
\end{split}
\right.
\end{equation}
in the space dimension $n \ge 2$. 
The vector displacement $u=u(t,x)$ at the point $(t,x)$, $x=(x_{1}, x_{2}, \cdots , x_{n})$ represents the $n$ dimensional vector function
$u=\mathstrut^{T}(u_{1}, u_{2}, \cdots, u_{n})$.
The vector functions  
$f_{j}(x)=\mathstrut^{T}(f_{j 1}, f_{j 2}, \cdots ,f_{j n})$ for $j=0,1$ denote the initial data.
Here and subsequently, 
``$\mathstrut^{T}\! A$'' stands for the transpose of the matrix $A$.
The Lam\'e constants $\lambda$ and $\mu$ satisfy 
\begin{equation} \label{eq:1.2}
\lambda +2 \mu>0, \qquad \mu >0, 
\end{equation}
which are independent of $t$ and $x$.
Our aim is to derive the upper and lower bound estimates of the $L^{2}$ norm of each component of the solution to \eqref{eq:1.1} $\| u_{j}(t) \|_{L^{2}(\mathbb{R}^{n})}$
for $j=1,2, \cdots, n$.
As a result we conclude the large time behavior of the solution to \eqref{eq:1.1} in the $L^{2}$.

An attempt in this direction has been made explicitly in \cite{I2} for the Cauchy problem of the wave equation: 
\begin{equation} \label{eq:1.3}
\left\{
\begin{split}
& \partial_{t}^{2} w -\alpha^{2} \Delta w =0, \quad t>0, \quad x \in \R^{n}, \\
& w(0,x)=w_{0}(x), \quad \partial_{t} w(0,x)=w_{1}(x) , \quad x \in \R^{n}, 
\end{split}
\right.
\end{equation}
where $n \ge 1$, $\alpha>0$ and $w=w(t,x):(0,\infty) \times \mathbb{R}^{n} \to \mathbb{R}$.
The problem \eqref{eq:1.3} corresponds to the problem \eqref{eq:1.1} with $\lambda+\mu=0$ when there is no interaction between the components of $u$.
One of the motivation of the estimate of the $L^{2}$ norm of $u$ is that it is 
essential for the sharp estimate of the local energy decay, which is still a major problem of the wave equation and the elastic wave equation with long history (cf. \cite{Morawetz}, \cite{Kawashita}, \cite{MST} and \cite{I2} and the references therein).
In particular, 
Ikehata \cite{I2} proved that
the sharp estimates of the $L^{2}$ norms of $w$ are proved as follows;
\begin{equation} \label{eq:1.4}
\begin{split}
C \left|\int_{\mathbb{R}} w_{1}(x) dx \right| \sqrt{t} \le \| w(t) \|_{L^{2}(\mathbb{R})}, \quad t \gg 1,
\end{split}
\end{equation}
\begin{equation} \label{eq:1.5}
\begin{split}
\| w(t) \|_{L^{2}(\mathbb{R})} 
\le C(\| w_{0} \|_{L^{2} \cap L^{1}(\mathbb{R})} +\| w_{1} \|_{L^{2} \cap L^{1}(\mathbb{R})}) \sqrt{t}, \quad t \ge 0
\end{split}
\end{equation}
for $n=1$ and 
\begin{equation} \label{eq:1.6}
\begin{split}
C\left|\int_{\mathbb{R}^{2}} w_{1}(x) dx \right| \sqrt{\log t} \le \| w(t) \|_{L^{2}(\mathbb{R}^{2})},
\quad t \gg 1,
\end{split}
\end{equation}
\begin{equation} \label{eq:1.7}
\begin{split}
\| w(t) \|_{L^{2}(\mathbb{R}^{2})} 
\le C(\| w_{0} \|_{L^{2} \cap L^{1}(\mathbb{R}^{2})} +\| w_{1} \|_{L^{2} \cap L^{1}(\mathbb{R}^{2})}) \sqrt{\log(t+2)}, \quad t \ge 0
\end{split}
\end{equation}
for $n=2$. 
In addition, the results in \cite{I2} have been extended in \cite{CheI1}:
\begin{equation} \label{eq:1.8}
\begin{split}
C \left|\int_{\mathbb{R}^{n}} w_{1}(x) dx \right| \le \| w(t) \|_{L^{2}(\mathbb{R}^{n})}, \quad t \gg 1,
\end{split}
\end{equation}
\begin{equation} \label{eq:1.9}
\begin{split}
\| w(t) \|_{L^{2}(\mathbb{R}^{n})} 
\le C(\| w_{0} \|_{L^{2} \cap L^{1}(\mathbb{R}^{n})} +\| w_{1} \|_{L^{2} \cap L^{1}(\mathbb{R}^{n})}),\quad t \ge 0
\end{split}
\end{equation}
for $n \ge 3$. 
Here we note that 
the estimates \eqref{eq:1.4}-\eqref{eq:1.9} precisely shows the large time behavior of 
$\| u(t) \|_{2}$. 
However, instead of $L^{1}$, 
a weighted space 
$L^{1,1}(\mathbb{R}^{n})
:= \{ f \in L^{1}(\mathbb{R}^{n};\mathbb{R})  (1+|x|) f\ \in L^{1}(\mathbb{R}^{n};\mathbb{R})  \}$ is used to obtain the lower bound estimates \eqref{eq:1.4},
\eqref{eq:1.6} and \eqref{eq:1.8}.
In the method developed by \cite{I2} and \cite{CheI1},  
it seems that the assumption $w_{1} \in L^{1,1}$ plays essential role.

On the other hand, 
Charao-Ikehata \cite{ChaI1} mentioned that the lower bound estimates of the $L^{2}$ norm are open and
obtained the upper estimate for the solution to \eqref{eq:1.1} with $n=2$ as follows; 
\begin{equation} \label{eq:1.10}
\begin{split}
& \| u(t) \|_{
\{ L^{2}(\mathbb{R}^{2}) \}^{2}
} \\
& \le C(\| f_{0} \|_{ 
\{ L^{2}(\mathbb{R}^{2}) \}^{2}
} 
+\| f_{1} \|_{ 
\{ L^{\infty}(\mathbb{R}^{2}) \}^{2}
} + \| f_{1} \|_{ 
\{ L^{1}(\mathbb{R}^{2}) \}^{2}
} \sqrt{\log(2L + (\lambda+2 \mu ) t)}
)
\end{split}
\end{equation}
for $t \gg 1$ under the assumption 
\begin{equation} \label{eq:1.11}
\begin{split}
\supp f_{0} \cup \supp f_{1} \subset \{ x \in \mathbb{R}^{2} : |x| \le L \} 
\end{split}
\end{equation}
for some $L>0$,
in the process of energy estimates for elastic waves with the space-variable coefficient damping term. 
Recently, Ikehata \cite{Ikehata} proved the estimates of the $L^{2}$ norm for the solution $u=\mathstrut^{T}(u_{1}, u_{2})$ to \eqref{eq:1.1} with $f_{0}=0$ and sufficiently smooth, rapidly decaying initial data $f_{1}$ when $n=2$:
\begin{equation} \label{eq:1.12}
\begin{split}
\| u(t) \|_{
\{ L^{2}(\mathbb{R}^{2}) \}^{2}
} & \ge \| u_{j}(t) \|_{
L^{2}(\mathbb{R}^{2}) 
} \ge C \left( \left|\int_{\mathbb{R}^{2}} f_{11}(x) dx \right| +\left|\int_{\mathbb{R}^{2}} f_{12}(x) dx \right| \right) \sqrt{\log t}, \quad t \gg 1, \\
\| u(t) \|_{
\{ L^{2}(\mathbb{R}^{2}) \}^{2} 
} 
& \le C(\| f_{1} \|_{ 
\{ L^{2}(\mathbb{R}^{2}) \}^{2}
} 
+\| f_{1} \|_{ 
\{ L^{1}(\mathbb{R}^{2}) \}^{2} 
} 
 \sqrt{\log (t+2)} ), \quad t \ge 0 
\end{split}
\end{equation}
for $j=1,2$. 
He also dealt with the situation that the elastic waves are decomposed into P-wave and S-wave.

In contrast to these previous studies, 
we are interested in the sharp estimates of the $L^{2}$ norm for each component under the minimal regularity of the initial data needed to prove it. 
To this end, 
we fix $N \in \{ 1,2, \cdots, n \}$ and consider the estimate of the $L^{2}$ norm of $u_{N}$, the $N$th component of $u$, for $(f_{0}, f_{1}) \in \{ H^{1} \}^{n} \times \{ L^{2} \}^{n}$.

To state our results precisely, we summarize the notation.  

For $\kappa \in \mathbb{Z}_{+} := \mathbb{N} \cup \{ 0 \}$ and $\gamma=
\mathstrut^{T}(\gamma_{1}, \gamma_{2}, \cdots , \gamma_{n}) \in \mathbb{Z}_{+}^{n}$ 
with $|\kappa|=\gamma := \displaystyle\sum_{k=1}^{n} \gamma_{k}$, 
we define the $\kappa$th moment of $f \in L^{1, \kappa}(\mathbb{R}^{n}; \mathbb{R})$ by 
\begin{equation*}
\begin{split}
m_{f, \gamma} := \frac{1}{\gamma! }\int_{\mathbb{R}^{n}} (-x)^{\gamma}f(x) dx,
\end{split}
\end{equation*} 
where 
$\gamma:= \gamma_{1}! \gamma_{2}! \cdots \gamma_{n}! $ and 
$(-x)^{\gamma}:= (-1)^{|\gamma|} x_{1}^{\gamma_{1}} x_{2}^{\gamma_{2}} 
\cdots x_{n}^{\gamma_{n}}$. 
For the simplicity, we denote 
$M_{j}:=\mathstrut^{T}(m_{j 1}, m_{j 2}, \cdots ,m_{j n})$
and 
\begin{equation*}
\begin{split}
P_{j k}:=  \mathstrut^{T}(p_{jk,1}, p_{jk,2}, \cdots ,p_{jk,n}) \in \mathbb{R}^{n}, 
\end{split}
\end{equation*}
where 
\begin{equation*}
\begin{split}
m_{jk} := m_{f_{jk}, 0},
\end{split}
\end{equation*}
\begin{equation*}
\begin{split}
p_{jk, \ell}:= m_{f_{jk}, \mathbf{e}_{\ell}}= -\int_{\mathbb{R}^{n}} x_{\ell} f_{jk}(x) dx 
\end{split}
\end{equation*}
for $j=0,1$ and $k,\ell=1,2, \cdots, n$.
Here $\mathbf{e}_{\ell}$ represents the unit vector in the $x_{\ell}$ direction. 
We also define
\begin{equation*} 
\begin{split}
|M_{j}|:=\sum_{k=1}^{n} |m_{j,k}|, \quad 
|P_{jk}|:= \sum_{\ell=1}^{n} |p_{jk, \ell}|
\end{split}
\end{equation*}
and
\begin{equation*} 
\begin{split}
|\mathbb{P}_{j,N}| :=
|P_{jN}|
+\sum_{k=1}^{n} |p_{jk,N}|
+
\sum_{k=1}^{n}  |p_{jk,k}| 
+ \sum_{ \substack{ k, \ell =1 \\ k< \ell \\
k \neq N, \ell \neq N} }^{n} |p_{jk, \ell} + p_{j \ell , k}| 
\end{split}
\end{equation*}
for $j=0,1$.
We denote by $L^{p}=L^{p}(\mathbb{R}^{n}; \mathbb{R})$ the usual $L^{p}$ space with the norm $\| \cdot \|_{L^{p}} := \| \cdot \|_{L^{p}(\mathbb{R}^{n})}$.
For the simplicity, we denote 
$\| f \|_{p}= \| f \|_{\{ L^{p} \}^{n}}$, 
where we define 
$\| f \|_{\{ L^{p} \}^{n}}:= \sum_{j=1}^{n} \| f_{j} \|_{L^{p}}$
for $f = \mathstrut^{T}(f_{1}, f_{2}, \cdots, f_{n})$ and $f_{j} \in L^{p}$, $j=1,2, \cdots, n$. 
Similarly, we denote $\| f \|_{1,1} =\| f \|_{\{ L^{1,1} \}^{n}}:= \sum_{j=1}^{n} \| f_{j} \|_{L^{1,1}}$.

Our first result indicates the sharp $L^{2}$ estimate of each component of the solution to \eqref{eq:1.1} as $t \to \infty$ for $(f_{0},  f_{1}) \in \{ H^{1} \cap L^{1} \}^{n} \times \{ L^{2} \cap L^{1} \}^{n}$.
\begin{thm} \label{thm:1.1}
Suppose that \eqref{eq:1.2} and 
$(f_{0},  f_{1}) \in \{ H^{1} \cap L^{1} \}^{n} \times \{ L^{2} \cap L^{1} \}^{n}$.
Then the solution of \eqref{eq:1.1}, $u$, belongs to 
$$
C([0, \infty); H^{1}) \cap C^{1}([0, \infty); L^{2}).
$$
In particular,
the $N$th component of solution $u_{N}$ satisfies the following estimates:\\
\begin{equation} \label{eq:1.13}
\begin{split}
\| u_{N}(t) \|_{L^{2}} \le \begin{cases}
& C( \| f_{0} \|_{2} + \| f_{1} \|_{1} \sqrt{\log (t+2)} + \| f_{1} \|_{2}), \quad n=2, \\
& C( \| f_{0} \|_{2} + \| f_{1} \|_{1}+ \| f_{1} \|_{2}), \quad n \ge 3
\end{cases}
\end{split}
\end{equation}
for $t \ge 0$. 
In addition, suppose that 
\begin{equation} \label{eq:1.14}
\begin{split}
\lambda + \mu \neq 0.
\end{split}
\end{equation}
If either $|M_{0}| \neq 0$ or $|M_{1}| \neq 0$, 
then 
\begin{equation} \label{eq:1.15}
\begin{split}
\| u_{N}(t) \|_{L^{2}} \ge \begin{cases}
& C |M_{0}|+C |M_{1}| \sqrt{\log (t+2)}, \quad n=2, \\
& C(|M_{0}| +|M_{1}|), \quad n \ge 3
\end{cases}
\end{split}
\end{equation}
for $t \gg 1$.
\end{thm}
Since the RHS of the estimate \eqref{eq:1.15} is independent of $N$,
we can rephrase the estimates \eqref{eq:1.13} and \eqref{eq:1.15} as follows: 
\begin{cor} \label{cor:1.2}
Assume that \eqref{eq:1.2}, \eqref{eq:1.14} and 
$(f_{0},  f_{1}) \in \{ H^{1} \cap L^{1} \}^{n} \times \{ L^{2} \cap L^{1} \}^{n}$.

\noindent
{\rm (i)} When $n=2$, 
if there exists $k_{0} \in \{1,2, \cdots, n \}$ such that 
$m_{1 k_{0}} \neq 0$, then 
\begin{equation} \label{eq:1.16}
\begin{split}
C |m_{1 k_{0}}| \sqrt{\log(t+2)} \le 
\| u_{j}(t) \|_{L^{2}} \le \| u(t) \|_{2} \le C( \| f_{0} \|_{2} + \| f_{1} \|_{1}  \sqrt{\log (t+2)} + \| f_{1} \|_{2})
\end{split}
\end{equation}
for all $j \in \{1,2 \}$ and $t \gg 1$.

\noindent
{\rm (ii)} When $n \ge 3$, 
if there exist $j_{0} \in \{0,1 \}$, $k_{0} \in \{1,2, \cdots, n \}$ such that 
$m_{j_{0} k_{0}} \neq 0$, then 
\begin{equation} \label{eq:1.17}
\begin{split}
C |m_{j_{0} k_{0}}| \le \| u_{j}(t) \|_{L^{2}} \le \| u(t) \|_{2} \le C( \| f_{0} \|_{2} + \| f_{1} \|_{1}+ \| f_{1} \|_{2})
\end{split}
\end{equation}
for all $j \in \{1,2, \cdots, n \}$ and $t \gg 1$.
\end{cor}
\begin{rem}
Both the RHS of \eqref{eq:1.16} and the one of \eqref{eq:1.17} imply the interaction between the components of the solution.
\end{rem}
Next we consider the $L^{2}$ estimates of solution to \eqref{eq:1.1} for $(f_{0},  f_{1}) \in \{ H^{1} \cap L^{1,1} \}^{n} \times \{ L^{2} \cap L^{1,1} \}^{n}$. 
This situation also corresponds to the lower bound estimate in the case $|M_{0}|=0$ and $|M_{1}|=0$.
\begin{thm} \label{thm:1.4}
Suppose that \eqref{eq:1.2} and 
$(f_{0},  f_{1}) \in \{ H^{1} \cap L^{1,1} \}^{n} \times \{ L^{2} \cap L^{1,1} \}^{n}$.
Then the following statements hold.
\begin{equation} \label{eq:1.18}
\begin{split}
\| u_{N}(t) \|_{2} \le  C( \| f_{0} \|_{2}+|M_{1}| \sqrt{\log(t+2)} +  \| f_{1} \|_{1,1} )
\end{split}
\end{equation}
for $n=2$ and $t \ge 0$.
Furthermore, suppose that \eqref{eq:1.14} and $t \gg 1$. \\
{\rm (i)} When $n=2$,
if $\displaystyle\sum_{j=0}^{1} \sum_{k=1}^{2}|P_{jk}| \neq 0$,
then 
\begin{equation} \label{eq:1.19}
\begin{split}
\| u_{N}(t) \|_{2} 
& \ge 
C (|M_{0}|+ |M_{1}| \sqrt{\log(t+2)}
+ \sum_{j=0}^{1} \sum_{k=1}^{2}|P_{jk}| ).
\end{split}
\end{equation}
{\rm (ii)} When $n \ge 3$, 
if either $|\mathbb{P}_{0,N}| \neq 0$ or $|\mathbb{P}_{1,N}| \neq 0$ holds,
then
\begin{equation} \label{eq:1.20}
\begin{split}
\| u_{N}(t) \|_{2} \ge C(|M_{0}|+ |M_{1}|+ |\mathbb{P}_{0,N}| +|\mathbb{P}_{1,N}|).
\end{split}
\end{equation}
\end{thm}
Comparing related results, our result have the advantages in the following points. 
One is the concerned with the regularity assumption of the initial data.
Due to Proposition 3.3 below, 
we can obtain a component-wise lower bound estimate even if 
$f_{0}$ and $f_{1}$ belong to $\{ L^{1} \}^{n}$ instead of $\{ L^{1,1} \}^{n}$. 
Therefore Theorem \ref{thm:1.1}
improves the estimate \eqref{eq:1.12} for the 2-d case in \cite{Ikehata} 
and completely answers 
to the $n$-d case problem for $n \ge 3$ proposed in \cite{ChaI1}.
Moreover, 
in the case of wave equations and elastic waves, 
when considering application to nonlinear problems, 
the estimates in $L^{2} \cap L^{1}$ seem to be easier to apply than $L^{2} \cap L^{1,1}$. 
For the nonlinear theory of the elastic waves, we refer to \cite{Sideris}, \cite{A}, \cite{HidanoZha} and \cite{Zha} and the references therein.
Second, under the assumption $f_{0}, f_{1} \in \{ L^{1,1} \}^{n}$, 
we obtain more detailed estimates in the form of a more explicit statement of how the components interact with each other.
From the estimates \eqref{eq:1.19} and \eqref{eq:1.20} in Theorem \ref{thm:1.4}, 
we see that there are structural differences in the assumptions between 2-d case and $n$-d case with $n \ge 3$.

Of course, 
our method is also applicable to the case of the wave equation \eqref{eq:1.3}. However, 
since the proof is almost similar, 
although some proofs are required, 
we omit them and only summarize the results in the appendix. 
This gives an improvement over previous studies \eqref{eq:1.4}, \eqref{eq:1.6} and \eqref{eq:1.8} and 
another proof of the Appendix in \cite{CT}.

The rest of the paper is organized as follows:
We introduce notations and mention auxiliary lemmas used below.
Section 3 is devoted to the proof of Proposition \ref{prop:3.3} which is needed to choose the initial data from $\{ H^{1} \cap L^{1} \}^{n} \times \{ L^{2} \cap L^{1} \}^{n}$ in the lower bound estimate instead of 
the weighted space $\{ H^{1} \cap L^{1,1} \}^{n} \times \{ L^{2} \cap L^{1,1} \}^{n}$. 
In section 4 and 5, we prove main results.
In particular, in section 4, we show Theorem \ref{thm:1.1} and in section 5, we deal with 
$(f_{0}, f_{1}) \in \{ H^{1} \cap L^{1,1} \}^{n} \times \{ L^{2} \cap L^{1,1} \}^{n}$ and we will show a more detailed dependence on the initial data in the lower bound estimate as Theorem \ref{thm:1.4}.
In the appendix, 
we are going to mention the estimates of the wave equation that corresponds to when the assumption $\lambda + \mu =0$.
\section{Preliminaries}
In this section, 
we shall summarize the notation and basic results that we will use throughout this paper.
\subsection{Solution formula and notations}
In this subsection, we recall the solution formula of the Cauchy problem \eqref{eq:1.1}.

Denoting by $M(\mathbb{R};n)$ the set of all $n \times n$ matrices over $\mathbb{R}$,
we define   
$I := {\rm diag}(1,1, \cdots 1) \in M(\mathbb{R};n)$ and
$P:=\frac{\xi}{|\xi|} \otimes \frac{\xi}{|\xi|} \in M(\mathbb{R};n)$.
We also define  
$\alpha_{1}:= \sqrt{\lambda+2 \mu}>0$ and 
$\alpha_{2}:= \sqrt{\mu}>0$ for the simplicity.
Then,  
we have the expression of the solution $u$ to \eqref{eq:1.1} in the Fourier space as follows;
\begin{equation*}
\begin{split}
\hat{u}(t,\xi) & = \cos(t \alpha_{1} |\xi|) P \hat{f}_{0} 
+ \frac{\sin(t \alpha_{1} |\xi|)}{\alpha_{1}|\xi|} P \hat{f}_{1} \\
& + \cos(t \alpha_{2} |\xi|) (I-P) \hat{f}_{0} 
+ \frac{\sin(t \alpha_{2} |\xi|)}{\alpha_{2}|\xi|} (I- P) \hat{f}_{1}.
\end{split}
\end{equation*}
Namely, for $N=1,2, \cdots, n$,
the $N$th component of $\hat{u}$, $\hat{u}_{N}$, can be represented in the following way;
\begin{equation} \label{eq:2.1}
\begin{split}
\hat{u}_{N}(t,\xi) & = \hat{V}_{N,c}(t) + \hat{V}_{N,s}(t),
\end{split}
\end{equation}
where 
\begin{equation*} 
\begin{split}
\hat{V}_{N,c}(t,\xi) & := \cos(t \alpha_{2}  |\xi|) \hat{f}_{0N} 
+ \left( \cos(t \alpha_{1} |\xi|) - \cos(t \alpha_{2} |\xi|) \right)
\frac{
\xi_{N} (\xi \cdot \hat{f}_{0})
}{
|\xi|^{2}
},  \\
\hat{V}_{N,s}(t, \xi) & :=\frac{\sin(t \alpha_{2} |\xi|)}{\alpha_{2}|\xi|} \hat{f}_{1N} 
+ \left(
\frac{\sin(t \alpha_{1} |\xi|)}{\alpha_{1}|\xi|} 
- \frac{\sin(t \alpha_{2} |\xi|)}{\alpha_{2}|\xi|} 
\right)
\frac{\xi_{N} (\xi \cdot \hat{f}_{1}) }{|\xi|^{2}}.
\end{split}
\end{equation*}
Here, 
``$\cdot$'' represents the inner product in $\mathbb{R}^{n}$,
$\hat{u}(t, \xi)=\mathstrut^{T}(\hat{u}_{1}, \hat{u}_{2}, \cdots, \hat{u}_{n})$
and 
$\hat{f}_{j}(\xi)=\mathstrut^{T}(\hat{f}_{j1}, \hat{f}_{j2}, \cdots ,\hat{f}_{jn})$ for $j=0,1$.
We decompose the function $\hat{u}_{N}$ into four parts. We do this in two ways:
\begin{equation} \label{eq:2.2}
\begin{split}
\hat{u}_{N}(t,\xi)
= \hat{U}_{N,c}(t,\xi) +\hat{U}_{N,s}(t,\xi) +\hat{R}_{N,c}(t,\xi) +\hat{R}_{N,s}(t,\xi),
\end{split}
\end{equation}
where
\begin{equation*} 
\begin{split}
\hat{U}_{N,c}(t,\xi) &:= 
\cos(t \alpha_{2} |\xi|) m_{0N} +\left( \cos(t \alpha_{1} |\xi|) - \cos(t \alpha_{2} |\xi|) \right)
\frac{
\xi_{N} (\xi \cdot M_{0})
}{
|\xi|^{2}
}, \\
\hat{U}_{N,s}(t,\xi)& := 
\frac{\sin(t \alpha_{2} |\xi|)}{\alpha_{2}|\xi|} m_{1N} +\left(
\frac{\sin(t \alpha_{1} |\xi|)}{\alpha_{1}|\xi|} 
- \frac{\sin(t \alpha_{2} |\xi|)}{\alpha_{2}|\xi|} 
\right)
\frac{\xi_{N} (\xi \cdot M_{1}) }{|\xi|^{2}}, 
\end{split}
\end{equation*}
\begin{equation*} 
\begin{split}
\hat{R}_{N,c}(t,\xi) & :=\hat{V}_{N,c}(t,\xi)-\hat{U}_{N,c}(t,\xi), \\
\hat{R}_{N,s}(t,\xi) &:=\hat{V}_{N,s}(t,\xi)-\hat{U}_{N,s}(t,\xi)
\end{split}
\end{equation*}
for the proof of Theorem \ref{thm:1.1} and 
\begin{equation} \label{eq:2.3}
\begin{split}
\hat{u}_{N}(t,\xi) = \left( \hat{U}_{N,c}(t,\xi) +
\hat{\tilde{U}}_{N,c}(t,\xi) \right) + \left( \hat{U}_{N,s}(t,\xi) + \hat{\tilde{U}}_{N,s}(t,\xi) 
\right)
+\hat{\tilde{R}}_{N,c}(t,\xi)  
+\hat{\tilde{R}}_{N,s}(t,\xi) ,
\end{split}
\end{equation}
where
\begin{equation*} 
\begin{split}
\hat{\tilde{U}}_{N,c}(t,\xi) &:= 
\cos(t \alpha_{2} |\xi|) (i \xi \cdot P_{0N}) 
+\left( \cos(t \alpha_{1} |\xi|) - \cos(t \alpha_{2} |\xi|) \right)
i \sum_{k=1}^{n} \left( \xi \cdot P_{0k} \right) \frac{\xi_{N} \xi_{k}}{|\xi|^{2}} \\
\hat{\tilde{U}}_{N, s}(t,\xi)& := 
\frac{\sin(t \alpha_{2} |\xi|)}{\alpha_{2}|\xi|} (i \xi \cdot P_{1N}) +\left(
\frac{\sin(t \alpha_{1} |\xi|)}{\alpha_{1}|\xi|} 
- \frac{\sin(t \alpha_{2} |\xi|)}{\alpha_{2}|\xi|} 
\right)
i \sum_{k=1}^{n} \left( \xi \cdot P_{1k} \right) \frac{\xi_{N} \xi_{k}}{|\xi|^{2}},
\end{split}
\end{equation*}
\begin{equation*} 
\begin{split}
\hat{\tilde{R}}_{N,c}(t,\xi) & :=\hat{R}_{N,c}(t,\xi)-\hat{\tilde{U}}_{N,c}(t,\xi) , \\
\hat{\tilde{R}}_{N,s}(t,\xi) & :=\hat{R}_{N,s}(t,\xi)-\hat{\tilde{U}}_{N,s}(t,\xi) 
\end{split}
\end{equation*}
for the proof of Theorem \ref{thm:1.4}.

We also introduce the auxiliary functions defined by 
\begin{equation*}
\begin{split}
W^{(\alpha, \beta)}_{c}(t,x) & := \mathcal{F}^{-1} 
\left[ 
e^{-\beta |\xi|^{2}} \cos(t \alpha |\xi|)
\right], \\
W^{(\alpha, \beta)}_{s}(t,x) & := \mathcal{F}^{-1} 
\left[ 
e^{-\beta |\xi|^{2}}
\frac{\sin(t \alpha |\xi|)}{\alpha |\xi|}
\right], \\
\tilde{W}^{(\alpha, \beta)}_{c,(N,k)}(t,x) & := \mathcal{F}^{-1} 
\left[ 
e^{-\beta |\xi|^{2}} \cos(t\alpha |\xi|) \frac{\xi_{N} \xi_{k}}{|\xi|^{2}}
\right], \\
\tilde{W}^{(\alpha, \beta)}_{s, (N,k)}(t,x) & := \mathcal{F}^{-1} 
\left[ 
e^{-\beta |\xi|^{2}}
\frac{\sin(t \alpha |\xi|)}{\alpha |\xi|}\frac{\xi_{N} \xi_{k}}{|\xi|^{2}}
\right],
\end{split}
\end{equation*}
where $\alpha>0$, $\beta>0$ and $N, k=1, 2, \cdots, n$.
\subsection{Useful formulas}
In this subsection we collect the well-known facts which are useful to prove the main results.
At first, we recall the properties of the spherical integrals.
\begin{lem} \label{Lem:2.1}
The following statements hold. 
\begin{align} \label{eq:2.4}
& \int_{\mathbb{S}^{n-1}} \omega^{\gamma} d \omega = 0, 
 \quad \gamma \in \mathbb{Z}_{+}^{n}, \ |\gamma|\ \text{is}\ \text{odd}, \\
& \int_{\mathbb{S}^{n-1}} d \omega = |\mathbb{S}^{n-1}|
= \dfrac{2 \pi^{\frac{n}{2}}}{\Gamma(\frac{n}{2})}, \label{eq:2.5}  \\
& \int_{\mathbb{S}^{n-1}}  \omega_{j}^{2} d \omega 
= \dfrac{1}{n} |\mathbb{S}^{n-1}|, \label{eq:2.6} \\
& \int_{\mathbb{S}^{n-1}}  \omega_{j}^{2} \omega_{k}^{2} d \omega 
= 
\begin{cases}
& \dfrac{3}{n(n+2)} |\mathbb{S}^{n-1}|, \quad j=k, \\
\ \\
& \dfrac{1}{n(n+2)} |\mathbb{S}^{n-1}|, \quad j \neq k
\end{cases} \label{eq:2.7}
\end{align}
and
\begin{align} 
& \int_{\mathbb{S}^{n-1}}  \omega_{j}^{2} \omega_{k}^{2} \omega_{\ell}^{2} d \omega 
= 
\begin{cases}
& \dfrac{15}{n(n+2)(n+4)} |\mathbb{S}^{n-1}|, \quad j=k=\ell, \\
\ \\
& \dfrac{3}{n(n+2)(n+4)} |\mathbb{S}^{n-1}|, \quad j = k, \, k \neq \ell, \\
\ \\
& \dfrac{1}{n(n+2)(n+4)} |\mathbb{S}^{n-1}|, \quad j \neq k, \, k \neq \ell, \, \ell \neq j
\end{cases} \label{eq:2.8}
\end{align}
for $j, k, \ell=1,2, \cdots, n$,
where $|\mathbb{S}^{n-1}|$ is the area of $\mathbb{S}^{n-1}$ in $\mathbb{R}^{n}$ and 
$$\Gamma(s) := \int_{0}^{\infty} e^{-x} x^{s-1} dx.$$
\end{lem}
The proof of Lemma \ref{Lem:2.1} is straightforward and we omit the detail.
\begin{lem} \label{Lem:2.2}
Let $n \ge 2$ be an integer.
Then it follows that
\begin{align} \label{eq:2.9}
\int_{\mathbb{S}^{n-1}} \omega_{k} (\omega \cdot M_{j}) d \omega 
=\frac{|\mathbb{S}^{n-1}|}{n} m_{jk}
=\dfrac{2 \pi^{\frac{n}{2}}}{n \Gamma(\frac{n}{2})} m_{jk}
\end{align}
and 
\begin{equation} \label{eq:2.10}
\begin{split}
\int_{\mathbb{S}^{n-1}}  (\omega \cdot M_{j})^{2}\omega_{k}^{2} d \omega
=
\frac{|\mathbb{S}^{n-1}|}{n(n+2)} \left( |M_{j}|^{2} +2 m_{jk}^{2} \right)
=
\frac{\pi^{\frac{n}{2}}}{\Gamma(\frac{n+4}{2})} \left( \frac{|M_{j}|^{2}}{2} +m_{jk}^{2} \right)
\end{split}
\end{equation}
for $j=0,1$ and $k=1,2, \cdots, n$.
\end{lem}
For the proof of Lemma \ref{Lem:2.2}, see e.g. \cite{T1}.

The following lemma is a direct consequence of the well-known fact, the Riemann-Lebesgue theorem (cf. \cite{Stein1}).
\begin{lem} \label{Lem:2.3}
Let $n \ge 1$ be an integer, $\gamma_{1}$, 
$\gamma_{2} \in \mathbb{R}$ with $\gamma_{1} \neq \gamma_{2}$ and $F \in L^{1}(\mathbb{R}^{n}; \mathbb{R})$. 
Then the following estimates hold;
\begin{equation} \label{eq:2.11}
\begin{split}
& \int_{\mathbb{R}^{n}} 
F(\xi) \cos(2 t \gamma_{1} |\xi|) d \xi=o(1), \\
&  \int_{\mathbb{R}^{n}} 
F(\xi) \sin(2 t \gamma_{1} |\xi|) d \xi=o(1), \\
&  \int_{\mathbb{R}^{n}} 
F(\xi) \cos(t \gamma_{1} r) \cos(t \gamma_{2} |\xi|) d \xi=o(1), \\
&  \int_{\mathbb{R}^{n}} 
F(\xi)\cos(t \gamma_{1} r) \sin(t \gamma_{2} |\xi|) d \xi=o(1), \\
&  \int_{\mathbb{R}^{n}} 
F(\xi) \sin(t \gamma_{1} r) \sin(t \gamma_{2} |\xi|) d \xi=o(1)
\end{split}
\end{equation}
as $t \to \infty$.
\end{lem}
%
\section{Estimates for $W^{(\alpha.\beta)}_{c}(t, x)$ and $W^{(\alpha.\beta)}_{s}(t, x)$}
In this section, 
we prove the approximation formulas of 
$W^{(\alpha, \beta)}_{c}(t) \ast f$ by $m_{f, \gamma} W^{(\alpha, \beta)}_{c}(t)$
and $W^{(\alpha, \beta)}_{s}(t) \ast f$ by $m_{f, \gamma} W^{(\alpha, \beta)}_{s}(t)$ as $t \to \infty$, 
for $n \ge 1$ and $\beta \gg 1$. 
Here we emphasize that thanks to the estimates \eqref{eq:3.7}, \eqref{eq:3.8}, \eqref{eq:3.10} and \eqref{eq:3.11} below, 
we can choose the initial data in Theorem \ref{thm:1.1} from $L^{2} \cap L^{1}$ instead of 
the weighted space $L^{2} \cap L^{1,1}$ for $n \ge 2$.
First,
we mention the case 
where the lower bounds of $W^{(\alpha, \beta)}_{c}(t)$ and $W^{(\alpha, \beta)}_{s}(t)$ 
are estimated by the constants independent of $t$, when $t \gg 1$.
%
\begin{lem} \label{Lem:3.1}
Let $n \in \mathbb{N}$, $k \in \mathbb{N}$, $\alpha>0$, 
$\beta>0$ and $\gamma \in \mathbb{Z}^{n}_{+}$ with $|\gamma| = k$. 
Then the following estimates hold.
\begin{equation} \label{eq:3.1}
\begin{split}
\| \partial_{x}^{\gamma} W^{(\alpha,\beta)}_{c} (t) \|_{L^{2}} 
= 2^{
-\frac{n}{4}-\frac{k}{2}-1
} 
\left( 
\int_{\mathbb{S}^{n-1} }
\omega^{2 \gamma} d \omega  \times
\Gamma 
\left(\frac{n}{2} + k \right)
\right)^{\frac{1}{2}} 
\beta^{-\frac{n}{4}-\frac{k}{2}} 
-o(1)
\end{split}
\end{equation}
for $(n,k) \in [1, \infty) \times [0,\infty)$ as $t \to \infty$
and 
\begin{equation} \label{eq:3.2}
\begin{split}
\| \partial_{x}^{\gamma} W^{(\alpha,\beta)}_{s} (t) \|_{L^{2}} 
= 2^{
-\frac{n}{4}-\frac{k}{2}-\frac{1}{2}
} 
\left( 
\int_{\mathbb{S}^{n-1} }
\omega^{2 \gamma} d \omega  \times
\Gamma 
\left(\frac{n}{2} + k-1 \right)
\right)^{\frac{1}{2}} 
\beta^{-\frac{n}{4}-\frac{k}{2}+\frac{1}{2}} 
-o(1)
\end{split}
\end{equation}
for $(n,k) \in [3, \infty) \times [0,\infty)$, $(n,k) \in \{ 2 \} \times [1,\infty)$ or 
$(n,k) \in \{ 1 \} \times [1,\infty)$,
as $t \to \infty$.
\end{lem}
\begin{proof}
%
We begin with the proof of \eqref{eq:3.1}.
Firstly we observe that 
$$
\int_{\mathbb{S}^{n-1} }
\omega^{2 \gamma} d \omega > 0
$$
for $\gamma \in \mathbb{Z}_{+}^{n}$ since the function $g(\omega) :=\omega^{2 \gamma}$ defined on $\mathbb{S}^{n-1}$ continuous and attains minimum of positive value.
See \cite{F} for more details on this discussion.
Noting that $\cos^{2}(z) = \dfrac{1+\cos(2z)}{2}$
and 
\begin{equation} \label{eq:3.3}
\begin{split}
\int_{0}^{\infty} e^{-2 \beta r^{2}} r^{k-1} dr =
2^{-\frac{k}{2}-1} \Gamma\left( \frac{k}{2} \right) \beta^{-\frac{k}{2}}
\end{split}
\end{equation}
for $k>0$,
we have  
\begin{equation*}
\begin{split}
\| \partial_{x}^{\gamma} W^{(\alpha, \beta)}_{c} (t) \|_{L^{2}}^{2} &= 
\int_{0}^{\infty} e^{-2 \beta r^{2}} \cos^{2} (t \alpha r) r^{n+ 2 |\gamma|-1} dr 
\int_{\mathbb{S}^{n-1} }
\omega^{2 \gamma} d \omega
\\
&= \frac{1}{2}
\left( 
\int_{0}^{\infty} e^{-2 \beta r^{2}} r^{n+2 |\gamma|-1} dr
+o(1)
\right) \int_{\mathbb{S}^{n-1} }
\omega^{2 \gamma} d \omega \\
& = 2^{-\frac{n}{2}-|\gamma|-2} 
\Gamma\left( \frac{n}{2} +|\gamma| \right) \beta^{-\frac{n}{2}-|\gamma|} 
\int_{\mathbb{S}^{n-1} }
\omega^{2 \gamma} d \omega
+o(1)
\end{split}
\end{equation*}
as $t \to \infty$,
by \eqref{eq:2.11}.
Thus we obtain \eqref{eq:3.1}.
We apply same argument to have \eqref{eq:3.2}.
We complete the proof of Lemma \ref{Lem:3.1}.
\end{proof}
Next, we discuss the case where $W^{(\alpha, \beta)}_{c}(t)$ and $W^{(\alpha, \beta)}_{s}(t)$ grow with respect to $t$. 
\begin{lem} \label{Lem:3.2}
Let $\alpha>0$, $\beta>0$ and $t \gg 1$. \\
{\rm(i)}\ When $n=1$, there exist constants $C_{1}>0$ and $C_{2}>0$ such that 
\begin{equation} \label{eq:3.4}
\begin{split}
C_{1} e^{-\beta} t^{\frac{1}{2}} \le \| W^{(\alpha,\beta)}_{s} (t) \|_{L^{2}} \le C_{2} t^{\frac{1}{2}}
\end{split}
\end{equation}
for $t >\dfrac{4 \pi}{3 \alpha}$. \\
\noindent
{\rm (ii)} When $n=2$, 
there exist constants $C_{3}>0$ and $C_{4}>0$ such that 
\begin{equation} \label{eq:3.5}
\begin{split}
C_{3} e^{-\beta} \sqrt{\log(t+2)} \le \| W^{(\alpha,\beta)}_{s} (t) \|_{L^{2}} 
\le C_{4} \sqrt{\log(t+2)}
\end{split}
\end{equation}
for $t \ge 2$.
Here $C_{j}$ is independent of $\beta$ for $j=1,2, 3,4$.
\end{lem}
\begin{proof}
At first, we prove the upper bound estimates in \eqref{eq:3.4} and \eqref{eq:3.5}. 
Using the fact that $|\frac{\sin(t |\xi|)}{|\xi|}| \le t$ in $\{ |\xi| \le t^{-\frac{1}{n}} \}$,
we get  
\begin{equation} \label{eq:3.6} 
\begin{split}
\| W^{(\alpha, \beta)}_{s} (t) \|_{L^{2}}^{2}  
& \le C t^{3-n} \int_{|\xi| \le t^{-\frac{1}{n}}} d \xi 
+ C \int_{t^{-\frac{1}{n}} \le |\xi| \le 1} |\xi|^{-2} d \xi 
+ C \int_{1\le |\xi| } e^{-2 \beta |\xi|^{2}} d \xi 
\\
&= C +C \int_{t^{-\frac{1}{n}}} r^{n-3} dr 
\le \begin{cases}
& C t, \quad n=1, \\
& C \log(t+2), \quad n=2.
\end{cases}
\end{split}
\end{equation}
For the lower bound estimate in \eqref{eq:3.4}, 
since $\sin^{2}( t \alpha r) \ge \frac{1}{2}$ for 
$r \in [\frac{\pi}{4 \alpha}t^{-1}, \frac{3\pi}{4 \alpha}t^{-1}]$, 
we see that 
\begin{equation} \label{eq:3.7}
\begin{split}
\| W^{(\alpha, \beta)}_{s} (t) \|_{L^{2}}^{2} & \ge \frac{1}{\alpha^{2}}
\int_{\frac{\pi}{4 \alpha}t^{-1}}^{1} e^{-2 \beta r^{2}} \sin^{2} (t \alpha r) r^{-2} dr \\
&= \frac{e^{-2 \beta}}{2 \alpha^{2}}
\int_{\frac{\pi}{4 \alpha}t^{-1}}^{\frac{3 \pi}{4 \alpha}t^{-1}} r^{-2} dr 
= \frac{4 e^{-2 \beta}}{\pi \alpha} t.
\end{split}
\end{equation}
Combining the estimates \eqref{eq:3.6} and \eqref{eq:3.7}, 
we have the desired estimate \eqref{eq:3.4}.
For $n=2$, 
noting that
\begin{equation*}
\begin{split}
\int_{t^{-\frac{1}{2}}}^{1} 
\cos(2 t \alpha r) 
r^{-1} dr = \left[ \frac{\sin(2 t \alpha r) }{2 t \alpha r} \right]_{r=t^{-\frac{1}{2}}}^{1} +
 \int_{t^{-\frac{1}{2}}}^{1} 
\frac{\sin(2 t\alpha r) }{2 t \alpha}
r^{-2} dr,
\end{split}
\end{equation*}
by the integral by parts, 
and then, we obtain 
\begin{equation} \label{eq:3.8}
\begin{split}
\left| \int_{t^{-\frac{1}{2}}}^{1} 
\cos(2 t \alpha r) 
r^{-1} dr \right| \le Ct^{-\frac{1}{2}} +C.
\end{split}
\end{equation}
We apply the fact that 
$\sin^{2} z= \frac{1-\cos (2z)}{2}$
to have
\begin{equation} \label{eq:3.9}
\begin{split}
\| W^{(\alpha, \beta)}_{s} (t) \|_{L^{2}}^{2} & \ge \frac{|\mathbb{S}^{1}|}{\alpha^{2}}
\int_{t^{-\frac{1}{2}}}^{1} e^{-2 \beta r^{2}} \sin^{2} (t \alpha r) r^{-1} dr \\
& \ge \frac{e^{-2 \beta}}{2 \alpha^{2}}
\int_{t^{-\frac{1}{2}}}^{1} \frac{1-\cos( 2 t \alpha r)  }{r} dr \\
& \ge C e^{-2 \beta} (\log(t+2) - t^{-\frac{1}{2}} -C).
\end{split}
\end{equation}
Therefore we obtain the desired estimate \eqref{eq:3.5} by \eqref{eq:3.6} and \eqref{eq:3.9}. 
We complete the proof of Lemma \ref{Lem:3.2}.
\end{proof}
The following proposition states the asymptotic behavior of $W^{(\alpha, \beta)}_{c}(t) \ast f$ 
and $W^{(\alpha, \beta)}_{s}(t) \ast f$ as $t \to \infty$, 
which plays essential role to show Theorems \ref{thm:1.1} and \ref{thm:1.4}. 
\begin{prop} \label{prop:3.3}
Let $n \in \mathbb{N}$ and $k \in \mathbb{N}$.
Suppose that $f \in L^{1,k}(\mathbb{R}^{n}; \mathbb{R})$.

\noindent
{\rm (i)}\ Under the assumption on Lemma \ref{Lem:3.1},
it holds that 
\begin{equation} \label{eq:3.10}
\begin{split}
& \left\| W^{(\alpha,\beta)}_{c} (t) \ast f - 
\sum_{  |\gamma| \le k} m_{f, \gamma} 
\partial_{x}^{\gamma} W^{(\alpha, \beta)}_{c} (t) \right\|_{L^{2}} 
+\left\| \tilde{W}^{(\alpha,\beta)}_{c} (t) \ast f - 
\sum_{  |\gamma| \le k} m_{f, \gamma} 
\partial_{x}^{\gamma} \tilde{W}^{(\alpha, \beta)}_{c} (t) \right\|_{L^{2}}
\\
& \le C(\beta^{-\frac{n-1}{4}-\frac{k+1}{2}} + \beta^{\frac{1}{4}} o(1)) \| f \|_{L^{1,k}}
+ C (\beta^{-\frac{n}{4}-\frac{k}{2}} +o(1))
\int_{|y| \ge \beta^{\frac{1}{4}} } 
|y|^{k} |f(y)| dy
\end{split}
\end{equation}
and 
\begin{equation} \label{eq:3.11}
\begin{split}
& \left\| W^{(\alpha,\beta)}_{s} (t) \ast f - 
\sum_{  |\gamma| \le k} m_{f, \gamma} 
\partial_{x}^{\gamma} W^{(\alpha, \beta)}_{s} (t) \right\|_{L^{2}} 
+\left\| \tilde{W}^{(\alpha,\beta)}_{s} (t) \ast f - 
\sum_{  |\gamma| \le k} m_{f, \gamma} 
\partial_{x}^{\gamma} \tilde{W}^{(\alpha, \beta)}_{s} (t) \right\|_{L^{2}}
\\
& \le C(\beta^{-\frac{n-3}{4}-\frac{k+1}{2}} + \beta^{\frac{1}{4}} o(1)) \| f \|_{L^{1,k}}
+ C (\beta^{-\frac{n-2}{4}-\frac{k}{2}} +o(1))
\int_{|y| \ge \beta^{\frac{1}{4}} } 
|y|^{k} |f(y)| dy
\end{split}
\end{equation}
as $t \to \infty$.\\
\noindent
{\rm (ii)}\ 
Under the assumption on Lemma \ref{Lem:3.2}, 
it holds that 
\begin{equation} \label{eq:3.12}
\begin{split}
\| W^{(\alpha, \beta)}_{s}(t) \ast f - m_{f,0} W^{(\beta, t)}_{s}(t) \|_{L^{2}} =o(t^{\frac{1}{2}}),
\end{split}
\end{equation}
as $t \to \infty$ for $n=1$ and 
\begin{equation} \label{eq:3.13}
\begin{split}
\| W^{(\alpha, \beta)}_{s}(t) \ast f - m_{f,0} W^{(\beta, t)}_{s}(t) \|_{L^{2}} =o(\sqrt{\log(t+2)}),
\end{split}
\end{equation}
\begin{equation} \label{eq:3.14}
\begin{split}
\| \tilde{W}^{(\alpha, \beta)}_{s, (N,k)}(t) \ast f - m_{f,0} \tilde{W}^{(\alpha, \beta)}_{s, (N,k)}(t) \|_{L^{2}} =o(\sqrt{\log(t+2)})
\end{split}
\end{equation}
as $t \to \infty$ for $n=2$.
\end{prop}
\begin{proof}
(i)\ 
We only prove the estimate \eqref{eq:3.11} for $k >0$. 
For $k=0$, we can apply the same argument with a minor modification.
At first, we decompose 
$$
\left\| W^{(\alpha,\beta)}_{s} (t) \ast f 
-
\sum_{  |\gamma| \le k} m_{f, \gamma} 
\partial_{x}^{\gamma} W^{(\alpha, \beta)}_{s}(t) \right\|_{L^{2}}$$ into three parts;
\begin{equation*}
\begin{split}
\left\| W^{(\alpha,\beta)}_{s} (t) \ast f 
-
\sum_{  |\gamma| \le k} m_{f, \gamma} 
\partial_{x}^{\gamma} W^{(\alpha, \beta)}_{s}(t) \right\|_{L^{2}}
\le \| I_{1}(t) \|_{L^{2}} +\| I_{2}(t) \|_{L^{2}} +\| I_{3}(t) \|_{L^{2}}
\end{split}
\end{equation*}
where 
\begin{equation*}
\begin{split}
I_{1}(t,x) & : = \int_{|y| \le \beta^{\frac{1}{4}}} 
\left(W^{(\alpha,\beta)}_{s} (t,x-y)-
\sum_{  |\gamma| \le k} m_{f, \gamma} 
\partial_{x}^{\gamma} W^{(\alpha, \beta)}_{s}(t,x) \right) f(y) dy, \\
I_{2}(t,x) & : = \int_{|y| \ge \beta^{\frac{1}{4}}} 
\left(W^{(\alpha,\beta)}_{s} (t,x-y)-
\sum_{  |\gamma| \le k-1} m_{f, \gamma} 
\partial_{x}^{\gamma} W^{(\alpha, \beta)}_{s}(t,x) \right) f(y) dy
\end{split}
\end{equation*}
and 
\begin{equation*}
\begin{split}
I_{3}(t,x) : = \frac{1}{\gamma !}\int_{|y| \ge \beta^{\frac{1}{4}}} (-y)^{\gamma} f(y) dy
\sum_{  |\gamma| = k} m_{f, \gamma} 
\partial_{x}^{\gamma} W^{(\alpha, \beta)}_{s}(t,x).
\end{split}
\end{equation*}
For the estimate of $I_{1}(t,x)$, 
we apply Taylor's theorem to have 
\begin{equation*}
\begin{split}
W^{(\alpha,\beta)}_{s} (t,x-y)-
\sum_{  |\gamma| \le k} m_{f, \gamma} 
\partial_{x}^{\gamma} W^{(\alpha, \beta)}_{s}(t,x)
=\sum_{  |\gamma|= k+1} \frac{(-y)^{\gamma}}{\gamma!}
\partial_{x}^{\gamma} W^{(\alpha, \beta)}_{s}(t,x-\theta_{1}y)
\end{split}
\end{equation*}
for some $\theta_{1} \in (0,1)$.
Thus we see that 
\begin{equation} \label{eq:3.15}
\begin{split}
\| I_{1}(t) \|_{L^{2}} 
& \le C \beta^{\frac{1}{4}} \sum_{ |\gamma| = k+1} 
\int_{|y| \le \beta^{\frac{1}{4}}
} |y|^{k} 
\| \partial_{x}^{\gamma} 
W^{(\alpha,\beta)}_{s} (t, \cdot- \theta_{1} y) 
\|_{L^{2}(\mathbb{R}^{n}_{x})} |f(y)| dy\\
& \le C(\beta^{-\frac{n-3}{4}-\frac{k+1}{2}} + \beta^{\frac{1}{4}} o(1)) \| f \|_{L^{1,k}}
\end{split}
\end{equation}
by \eqref{eq:3.2}.
Similarly, noting that
\begin{equation*}
\begin{split}
W^{(\alpha,\beta)}_{s} (t,x-y)-
\sum_{  |\gamma| \le k-1} m_{f, \gamma} 
\partial_{x}^{\gamma} W^{(\alpha, \beta)}_{s}(t,x)
=\sum_{  |\gamma|= k} \frac{(-y)^{\gamma}}{\gamma!}
\partial_{x}^{\gamma} W^{(\alpha, \beta)}_{s}(t,x-\theta_{2}y)
\end{split}
\end{equation*}
for some $\theta_{2} \in (0,1)$,
we obtain
\begin{equation} \label{eq:3.16}
\begin{split}
\| I_{2}(t) \|_{L^{2}} 
& \le C \sum_{  |\gamma| =k}
\int_{|y| \ge \beta^{\frac{1}{4}}
}  
|y|^{k}
\| \partial_{x}^{\gamma} W^{(\alpha,\beta)}_{s} (t, \cdot- \theta y) \|_{L^{2}(\mathbb{R}^{n}_{x})}  |f(y)| dy \\
& \le C (\beta^{-\frac{n-2}{4}-\frac{k}{2}} +o(1))
\int_{|y| \ge \beta^{\frac{1}{4}} } 
|y|^{k} |f(y)| dy 
\end{split}
\end{equation}
as $t \to \infty$.
Moreover the direct calculation and \eqref{eq:3.2} give
\begin{equation} \label{eq:3.17}
\begin{split}
\| I_{3}(t) \|_{L^{2}} 
 \le C (\beta^{-\frac{n-2}{4}-\frac{k}{2}} +o(1))
\int_{|y| \ge \beta^{\frac{1}{4}} } 
|y|^{k} |f(y)| dy 
\end{split}
\end{equation}
as $t \to \infty$.
Combining the estimates \eqref{eq:3.15}-\eqref{eq:3.17},
and
\begin{equation} \label{eq:3.18}
\begin{split}
& \left\| \tilde{W}^{(\alpha,\beta)}_{s} (t) \ast f - 
\sum_{  |\gamma| \le k} m_{f, \gamma} 
\partial_{x}^{\gamma} \tilde{W}^{(\alpha, \beta)}_{s} (t) \right\|_{L^{2}} \\
& 
\le \left\| W^{(\alpha,\beta)}_{s} (t) \ast f - 
\sum_{  |\gamma| \le k} m_{f, \gamma} 
\partial_{x}^{\gamma} W^{(\alpha, \beta)}_{s} (t) \right\|_{L^{2}},
\end{split}
\end{equation}
we arrive at the desired estimate \eqref{eq:3.11}.
The estimate \eqref{eq:3.10} is shown in the same manner by \eqref{eq:3.1}.

\noindent
(ii) \
At first we prove the estimate \eqref{eq:3.12}. 
For $n=1$, we decompose the integrand in the LHS of \eqref{eq:3.12} as follows; 
\begin{equation*}
\begin{split}
\| W^{(\alpha, \beta)}_{s} (t) \ast f -m_{f,0}  W^{(\alpha, \beta)}_{s} (t) \|_{L^{2}} 
\le \| I_{4}(t) \|_{L^{2}} +\| I_{5}(t) \|_{L^{2}},
\end{split}
\end{equation*}
where 
\begin{equation*}
\begin{split}
I_{4}(t,x) & : = \int_{|y| \le t^{\frac{1}{4}}} 
(W^{(\alpha,\beta)}_{s} (t,x-y)- W^{(\alpha, \beta)}_{s}(t,x)) f(y) dy,
\end{split}
\end{equation*}
and 
\begin{equation*}
\begin{split}
I_{5}(t,x) & : = \int_{|y| \ge t^{\frac{1}{4}}} 
(W^{(\alpha,\beta)}_{s} (t,x-y)-
W^{(\alpha, \beta)}_{s}(t,x)) f(y) dy.
\end{split}
\end{equation*}
We estimate $I_{4}(t,x)$.
As in the estimate for $\| I_{1}(t) \|_{L^{2}}$,  
we have
\begin{equation*}
\begin{split}
W^{(\alpha,\beta)}_{s} (t,x-y)- W^{(\alpha, \beta)}_{s}(t,x)
=-y
\partial_{x} W^{(\alpha, \beta)}_{s}(t,x-\theta_{2}y)
\end{split}
\end{equation*}
for some $\theta_{2} \in (0,1)$ and then
\begin{equation} \label{eq:3.19}
\begin{split}
\| I_{4}(t) \|_{L^{2}} 
& \le C t^{\frac{1}{4}} 
\int_{|y| \le t^{\frac{1}{4}}
}
\| \partial_{x}
W^{(\alpha,\beta)}_{s} (t, \cdot- \theta_{2} y) 
\|_{L^{2}(\mathbb{R}_{x})} |f(y)| dy\\
& \le C(t^{\frac{1}{4}} + o(1)) \| f \|_{L^{1}}
\end{split}
\end{equation}
as $t \to \infty$.
It is easy to see that
\begin{equation} \label{eq:3.20}
\begin{split}
\| I_{5}(t) \|_{L^{2}} 
 \le C t^{\frac{1}{2}} 
\int_{|y| \ge t^{\frac{1}{4}} } |f(y)| dy 
\end{split}
\end{equation}
as $t \to \infty$.
Combining the estimates \eqref{eq:3.19} and \eqref{eq:3.20},
we obtain the desired estimate \eqref{eq:3.12}.


Next, we prove \eqref{eq:3.13}.
As in the proof of \eqref{eq:3.11} with minor modification, we decompose the integrand as follows; 

\begin{equation} \label{eq:3.21}
\begin{split}
\| W^{(\alpha, \beta)}_{s} (t) \ast f -m_{f,0}  W^{(\alpha, \beta)}_{s} (t) \|_{L^{2}} 
\le \| I_{6}(t) \|_{L^{2}} +\| I_{7}(t) \|_{L^{2}},
\end{split}
\end{equation}
where 
\begin{equation*}
\begin{split}
I_{6}(t,x) & : = \int_{|y| \le (\log(t+2))^{\frac{1}{4}}} 
(W^{(\alpha,\beta)}_{s} (t,x-y)- W^{(\alpha, \beta)}_{s}(t,x)) f(y) dy,
\end{split}
\end{equation*}
and 
\begin{equation*}
\begin{split}
I_{7}(t,x) & : = \int_{|y| \ge (\log(t+2))^{\frac{1}{4}} } 
(W^{(\alpha,\beta)}_{s} (t,x-y)-W^{(\alpha, \beta)}_{s}(t,x)) f(y) dy.
\end{split}
\end{equation*}
Once we have \eqref{eq:3.21}, the same argument in the proof of \eqref{eq:3.11} leads to
\begin{equation} \label{eq:3.22}
\begin{split}
\| I_{6}(t) \|_{L^{2}} 
& \le C( (\log(t+2))^{\frac{1}{4}} + o(1)) \| f \|_{L^{1}}
\end{split}
\end{equation}
and 
\begin{equation} \label{eq:3.23}
\begin{split}
\| I_{7}(t) \|_{L^{2}} 
 \le C \sqrt{\log(t+2)}
\int_{|y| \ge (\log(t+2))^{\frac{1}{4}}  } |f(y)| dy 
\end{split}
\end{equation}
as $t \to \infty$.
Summing up the estimates \eqref{eq:3.21}-\eqref{eq:3.23},
we have the desired estimate \eqref{eq:3.13}.
Noting \eqref{eq:3.18}, we also have \eqref{eq:3.14} immediately.
We complete the proof of Proposition \ref{prop:3.3}.
\end{proof}
%
\section{Proof of Theorem \ref{thm:1.1}}
In this section, we prove Theorem \ref{thm:1.1}. 
To this end, 
we show the estimate \eqref{eq:1.15} 
under the assumption on \eqref{eq:1.14}.
Indeed, once we have the expression of $\hat{u}_{N}(t)$, \eqref{eq:2.3}, 
we easily prove the upper bound estimate \eqref{eq:1.13} 
and we omit the detail.
\subsection{Lower bounds for $n=2$}
In this subsection, we show the estimate \eqref{eq:1.15}.
At first, 
it is easy to see that
\begin{equation} \label{eq:4.1}
\begin{split}
\| u_{N}(t) \|_{2} & \ge 
\| \hat{U}_{N,s}(t)+\hat{R}_{N,s}(t)) \|_{2}
-\|  \hat{U}_{N,c}(t)+\hat{R}_{N,c}(t) \|_{2} \\
& \ge \| e^{-\beta |\xi|^{2} } \hat{U}_{N,s}(t) \|_{2}
-\| \hat{U}_{N,c}(t)+\hat{R}_{N,c}(t) \|_{2}
-\| e^{-\beta |\xi|^{2} } \hat{R}_{N,s}(t) \|_{2}
\end{split}
\end{equation}
and 
\begin{equation} \label{eq:4.2}
\begin{split}
\| \hat{U}_{N,c}(t)+\hat{R}_{N,c}(t) \|_{2}
& \le 
\left\| 
\cos(t \alpha_{2} |\xi|) \hat{f}_{0N} \right\|_{2} 
+\left\| \left( |\cos(t \alpha_{1} |\xi|)|+ |\cos(t \alpha_{2} |\xi|)| \right)
\frac{
\xi_{N} (\xi \cdot \hat{f}_{0})
}{
|\xi|^{2}
} \right\|_{2} \\
& \le 2 \| f_{0} \|_{2}
\end{split}
\end{equation}
by \eqref{eq:2.3}.
Now we estimate $\| e^{-\beta |\xi|^{2} } \hat{U}_{N,s}(t) \|_{2}$ in \eqref{eq:4.1}.
Noting that $|\mathbb{S}^{1}|=2 \pi$, \eqref{eq:2.9} and \eqref{eq:2.10},
we see that 
\begin{equation} \label{eq:4.3}
\begin{split}
\| e^{-\beta |\xi|^{2} } \hat{U}_{N,s}(t) \|_{2}^{2} 
& = \frac{\pi |M_{1}|^{2}}{4}  
\int_{0}^{\infty} e^{-2 \beta r^{2}}
\left( \frac{\sin(t \alpha_{1} r)}{\alpha_{1}}
-
\frac{\sin(t \alpha_{2} r)}{\alpha_{2}} \right)^{2}r^{-1} dr\\
& +\frac{\pi m_{1N}^{2}}{2}  
\int_{0}^{\infty} e^{-2 \beta r^{2}}
\left( \frac{\sin(t \alpha_{1} r)}{\alpha_{1}}
+
\frac{\sin(t \alpha_{2} r)}{\alpha_{2}} \right)^{2}r^{-1} dr.
\end{split}
\end{equation}
Thus for $t >2$, 
noting that $\sin^{2} z= \frac{1-\cos (2z)}{2}$ and 
$\sin z_{1} \sin z_{2} =-\dfrac{1}{2}
\left( 
\cos(z_{1}+z_{2}) - \cos(z_{1}-z_{2})
\right)$,
we have 
\begin{equation} \label{eq:4.4}
\begin{split}
& \int_{0}^{\infty} e^{-2 \beta r^{2}}
\left( \frac{\sin(t \alpha_{1} r)}{\alpha_{1}}
-
\frac{\sin(t \alpha_{2} r)}{\alpha_{2}} \right)^{2} r^{-1} dr \\
& \ge e^{-2 \beta}
\int_{t^{-\frac{1}{2}}}^{1} 
\left( \frac{\sin(t \alpha_{1} r)}{\alpha_{1}}
-
\frac{\sin(t \alpha_{2} r)}{\alpha_{2}} \right)^{2}r^{-1} dr \\
& \ge \frac{ e^{-2 \beta}}{2}
\left(
\frac{1}{\alpha_{1}^{2}} 
+
\frac{1}{\alpha_{2}^{2}}
\right)  
\int_{t^{-\frac{1}{2}}}^{1} r^{-1} dr -C \left| \int_{t^{-\frac{1}{2}}}^{1} 
\left\{ \cos(2 t \alpha_{1} r) - \cos(2 t \alpha_{2} r) 
\right\}
r^{-1} dr \right| \\
& 
-C \left| \int_{t^{-\frac{1}{2}}}^{1} 
\left\{\cos( t (\alpha_{1} +\alpha_{2}) r) 
+
\cos( t (\alpha_{1} -\alpha_{2}) r)
\right\}
r^{-1} dr \right|.
\end{split}
\end{equation}
By the same way as in \eqref{eq:3.8}, 
we also have 
\begin{equation} \label{eq:4.5}
\begin{split}
& \left| \int_{t^{-\frac{1}{2}}}^{1} 
\cos(2 t \alpha_{2} r) 
r^{-1} dr \right| \le Ct^{-\frac{1}{2}} +C, \\
& 
\left| \int_{t^{-\frac{1}{2}}}^{1} 
\cos(t(\alpha_{1} + \alpha_{2}) r) 
r^{-1} dr \right|  +
\left| \int_{t^{-\frac{1}{2}}}^{1} 
\cos(t(\alpha_{1} - \alpha_{2}) r) 
r^{-1} dr \right| \le Ct^{-\frac{1}{2}} +C. 
\end{split}
\end{equation}
Summing up \eqref{eq:4.3}-\eqref{eq:4.5},
we obtain
\begin{equation} \label{eq:4.6}
\begin{split}
\| e^{-\beta |\xi|^{2} } \hat{U}_{N,s}(t) \|_{2}^{2} 
& \ge C_{0} \frac{e^{-\beta} \pi |M_{1}|^{2}}{8}
\left(
\frac{1}{\alpha_{1}^{2}} 
+
\frac{1}{\alpha_{2}^{2}}
\right) \log(t+2) -C
\end{split}
\end{equation}
for $t \gg 1$.
Finally we show the estimate for $\| e^{-\beta |\xi|^{2} } \hat{R}_{N,s}(t) \|_{2}$ in \eqref{eq:4.1}.
By \eqref{eq:3.13} and \eqref{eq:3.14}, 
we see that
\begin{equation} \label{eq:4.7}
\begin{split}
& \| e^{-\beta |\xi|^{2} } \hat{R}_{N,s}(t) \|_{2} \\
& \le
\| W^{(\alpha_{2},\beta)}_{s} (t) \ast f_{1N} - m_{1N} W^{(\alpha_{2},\beta)}_{s} (t) \|_{2} 
+ \sum_{\ell=1}^{2} \sum_{j=1}^{2}
\| 
\tilde{W}^{(\alpha_{\ell},\beta)}_{s} (t) \ast f_{1j} - 
m_{1j} \tilde{W}^{(\alpha_{\ell},\beta)}_{s} (t) \|_{2} \\
& =o(\sqrt{\log(t+2)}),
\end{split}
\end{equation}
as $t \to \infty$.
Combining the estimates \eqref{eq:4.1}, \eqref{eq:4.2}, \eqref{eq:4.6} and \eqref{eq:4.7}
gives the desired estimate \eqref{eq:1.15}.
\subsection{Lower bounds for $n \ge 3$}
We firstly show the case $|M_{0}| \neq 0$ and $|M_{1}| \neq 0$.
We begin with the observation that
\begin{equation} \label{eq:4.8}
\begin{split}
\| u_{N}(t) \|_{2} & \ge \| e^{-t |\xi|^{2} } \hat{u}_{N}(t) \|_{2} \\
& \ge \| e^{-\beta |\xi|^{2} } ( \hat{U}_{N,c}(t)+\hat{U}_{N,s}(t) ) \|_{2}
-\| e^{-\beta |\xi|^{2} } \hat{R}_{N,c}(t) \|_{2}
-\| e^{-\beta |\xi|^{2} } \hat{R}_{N,s}(t) \|_{2}
\end{split}
\end{equation}
by $0< e^{-\beta |\xi|^{2}} \le 1$.
Now we estimate $\| e^{-\beta |\xi|^{2} } ( \hat{U}_{N,c}(t)+\hat{U}_{N,s}(t) ) \|_{2}^{2}$. 
The direct calculation gives 
\begin{equation} \label{eq:4.9}
\begin{split}
\| e^{-\beta |\xi|^{2} } ( \hat{U}_{N,c}(t)+\hat{U}_{N,s}(t) ) \|_{2}^{2}
& = \| e^{-\beta |\xi|^{2} } \hat{U}_{N,c}(t) \|_{2}^{2}
+
\| e^{-\beta |\xi|^{2} } \hat{U}_{N,s}(t)  \|_{2}^{2} \\
& + 2 \int_{\mathbb{R}^{n}} e^{-2\beta |\xi|^{2}} \hat{U}_{N,c}(t, \xi) \hat{U}_{N,s}(t, \xi) d \xi
\end{split}
\end{equation}
and 
\begin{equation} \label{eq:4.10}
\begin{split}
& \int_{\mathbb{R}^{n}} e^{-2\beta |\xi|^{2}} \hat{U}_{N,c}(t, \xi) \hat{U}_{N,s}(t, \xi) d \xi =o(1)
\end{split}
\end{equation}
as $t \to \infty$ by \eqref{eq:2.11}.
Here we apply \eqref{eq:2.9}, \eqref{eq:2.10} and \eqref{eq:2.11} to have
\begin{equation*}
\begin{split}
\| e^{-\beta |\xi|^{2} } \hat{U}_{N,c}(t) \|_{2}^{2} 
& = \frac{m_{0N}^{2}}{ \alpha_{2}^{2}} |\mathbb{S}^{n-1}| \left(1-\frac{2}{n} \right)
\int_{0}^{\infty} e^{-2 \beta r^{2}} \cos^{2}(t \alpha_{2} r)r^{n-1} dr \\
& + \frac{ |\mathbb{S}^{n-1}| |M_{0}|^{2} }{n(n+2)}
\int_{0}^{\infty}
e^{-2 \beta r^{2}} 
\left( 
\cos^{2}(t \alpha_{1} r)
+
\cos^{2}(t \alpha_{2} r)
\right) r^{n-1} dr +o(1) \\
& \ge \frac{ |M_{0}|^{2} }{n(n+2)} 
(\| W^{(\alpha_{1}, \beta)}_{c} (t) \|_{2}^{2}+\| W^{(\alpha_{2}, \beta)}_{c} (t) \|_{2}^{2}) +o(1),
\end{split}
\end{equation*}
as $t \to \infty$.
Namely, we have
\begin{equation} \label{eq:4.11}
\begin{split}
\| e^{-\beta |\xi|^{2} } \hat{U}_{N,c}(t) \|_{2}
& \ge |M_{0}|
\left(
\frac{
|\mathbb{S}^{n-1}| }{n(n+2)} 
2^{-\frac{n}{2}-1} \Gamma\left( \frac{n}{2} \right) \right)^{\frac{1}{2}} \beta^{-\frac{n}{4}}
\end{split}
\end{equation}
by \eqref{eq:3.1} for $t \gg 1$.
We apply the same argument to have
\begin{equation*}
\begin{split}
\| e^{-\beta |\xi|^{2} } \hat{U}_{N,s}(t) \|_{2}^{2} 
& \ge \frac{ |M_{1}|^{2} }{n(n+2)} 
(\| W^{(\alpha_{1}, \beta)}_{s} (t) \|_{2}^{2}+\| W^{(\alpha_{2}, \beta)}_{s} (t) \|_{2}^{2}) +o(1) \\
\end{split}
\end{equation*}
as $t \to \infty$ and then
\begin{equation} \label{eq:4.12}
\begin{split}
\| e^{-\beta |\xi|^{2} } \hat{U}_{N,s}(t) \|_{2}
& \ge |M_{1}|
\left(
\frac{
|\mathbb{S}^{n-1}| }{n(n+2)} 
2^{-\frac{n}{2}} \Gamma\left( \frac{n}{2}-1 \right) \right)^{\frac{1}{2}} \beta^{-\frac{n-2}{4}}
\end{split}
\end{equation}
for $t \gg 1$ by \eqref{eq:3.2}.
Finally we show the estimate of $\| e^{-\beta |\xi|^{2} } \hat{R}_{N,c}(t) \|_{2}^{2}$.
Here we note the fact that
\begin{equation} \label{eq:4.13}
\begin{split}
\lim_{\beta \to \infty}
\sum_{k=1}^{n} \int_{|y| \ge \beta^{\frac{1}{4}}} |f_{jk}(y)|dy=0
\end{split}
\end{equation}
for $j=0,1$,
since $f_{0} \in \{ L^{1} \}^{n}$ and $f_{1} \in \{ L^{1} \}^{n}$.
Then applying \eqref{eq:3.10} with $k=0$ and \eqref{eq:4.13},
we see that there exists $\beta_{0}>0$ such that 
\begin{equation} \label{eq:4.14}
\begin{split}
& \| e^{-\beta |\xi|^{2} } \hat{R}_{N,c}(t) \|_{2} \\
& \le 
\| W^{(\alpha_{2}, \beta)}_{c}(t) \ast f_{0N}-m_{0N} W^{(\alpha_{2}, \beta)}_{c}(t) \|_{2}  
+\sum_{\ell=1}^{2} \sum_{k=1}^{n} 
\left\| \tilde{W}^{(\alpha_{\ell}, \beta)}_{c,(N,k)}(t) \ast f_{0k}
-m_{0k} \tilde{W}^{(\alpha_{\ell}, \beta)}_{c, (N,k)}(t) \right\|_{2} \\
& \le C\beta^{-\frac{n}{4}-\frac{1}{4}} \| f_{0} \|_{1}
+C \beta^{-\frac{n}{4}}\sum_{k=1}^{n} \int_{|y| \ge \beta^{\frac{1}{4}}} |f_{0k}(y)|dy +o(1) \\
& \le \frac{|M_{0}|}{2} 
\left(
\frac{
|\mathbb{S}^{n-1}| }{n(n+2)} 
2^{-\frac{n}{2}-1} \Gamma\left( \frac{n}{2} \right) \right)^{\frac{1}{2}} \beta^{-\frac{n}{4}}
\end{split}
\end{equation}
for $\beta \ge \beta_{0}$ and $t \gg 1$.
Similarly, there exists $\beta_{1}>0$ such that
\begin{equation} \label{eq:4.15}
\begin{split}
& \| e^{-\beta |\xi|^{2} } \hat{R}_{N,s}(t) \|_{2} \\
&\le 
C \beta^{-\frac{n-1}{4}} \| f_{1} \|_{1}
+C\beta^{-\frac{n-2}{4}}
\sum_{k=1}^{n} \int_{|y| \ge \beta^{\frac{1}{4}}} |f_{1k}(y)|dy +o(1) \\
& \le \frac{|M_{1}|}{2}
\left(
\frac{
|\mathbb{S}^{n-1}| }{n(n+2)} 
2^{-\frac{n}{2}} \Gamma\left( \frac{n}{2}-1 \right) \right)^{\frac{1}{2}} \beta^{-\frac{n-2}{4}},
\end{split}
\end{equation}
for $\beta \ge \beta_{1}$ and $t \gg 1$ by \eqref{eq:3.11} with $k=0$ and \eqref{eq:4.13}.
Therefore taking $\beta \ge \max \{ \beta_{0}, \beta_{1} \}$, and combining the estimates 
\eqref{eq:4.8}-\eqref{eq:4.12}, \eqref{eq:4.14} and \eqref{eq:4.15},
we obtain 
\begin{equation*}
\begin{split}
\| u_{N}(t) \|_{2} & \ge 
\frac{|M_{0}|}{2} 
\left(
\frac{
|\mathbb{S}^{n-1}| }{n(n+2)} 
2^{-\frac{n}{2}-1} \Gamma\left( \frac{n}{2} \right) \right)^{\frac{1}{2}} \beta^{-\frac{n}{4}} \\
& +\frac{|M_{1}|}{2}
\left(
\frac{
|\mathbb{S}^{n-1}| }{n(n+2)} 
2^{-\frac{n}{2}} \Gamma\left( \frac{n}{2}-1 \right) \right)^{\frac{1}{2}} \beta^{-\frac{n-2}{4}}
+o(1) 
\end{split}
\end{equation*}
as $t \to \infty$, 
which implies the desired estimate \eqref{eq:1.15} for $|M_{0}| \neq 0$ and $|M_{1}| \neq 0$.
On the other hand, if $|M_{0}| = 0$ and $|M_{1}| \neq 0$, 
the estimates \eqref{eq:3.10} and \eqref{eq:4.13} yield that 
there exists $\tilde{\beta}_{0}>0$ such that
\begin{equation*}
\begin{split}
& \| e^{-\beta |\xi|^{2} } \hat{R}_{N,c}(t) \|_{2} \le \frac{|M_{1}|}{4} 
\left(
\frac{
|\mathbb{S}^{n-1}| }{n(n+2)} 
2^{-\frac{n}{2}} \Gamma\left( \frac{n}{2}-1 \right) \right)^{\frac{1}{2}} \beta^{-\frac{n-2}{4}}
\end{split}
\end{equation*}
for $\beta \ge \tilde{\beta}_{0}$ and $t \gg 1$.
Then the same argument in the case $|M_{0}| \neq 0$ and $|M_{1}| \neq 0$ gives 
\begin{equation*}
\begin{split}
\| u_{N}(t) \|_{2} & \ge \frac{|M_{1}|}{4}
\left(
\frac{
|\mathbb{S}^{n-1}| }{n(n+2)} 
2^{-\frac{n}{2}} \Gamma\left( \frac{n}{2}-1 \right) \right)^{\frac{1}{2}} \beta^{-\frac{n-2}{4}}
+o(1) 
\end{split}
\end{equation*}
as $t \to \infty$, which is the desired estimate.
If $|M_{0}| \neq 0$ and $|M_{1}| = 0$,
we note that 
\begin{equation*}
\begin{split}
\int_{\mathbb{R}^{n}} e^{-2 \beta |\xi|^{2} } 
(\hat{V}_{N,c}(t)\bar{\hat{V}}_{N,s}(t)+\hat{V}_{N,c}(t) \bar{\hat{V}}_{N,s}(t)) d \xi 
 =o(1)
\end{split}
\end{equation*}
as $t \to \infty$ by \eqref{eq:2.11} and 
\begin{equation*}
\begin{split}
e^{-2 \beta |\xi|^{2} } \frac{\hat{f}_{0N} \bar{\hat{f}}_{1 N}}{|\xi|},  
e^{-2 \beta |\xi|^{2} } \hat{f}_{0N} \frac{\xi_{N} (\xi \cdot \bar{\hat{f}}_{1})}{|\xi|^{3}}, 
e^{-2 \beta |\xi|^{2} } \hat{f}_{1N} \frac{\xi_{N} (\xi \cdot \bar{\hat{f}}_{0})}{|\xi|^{3}}, 
e^{-2 \beta |\xi|^{2} } \frac{\xi_{N}^{2} (\xi \cdot \hat{f}_{0})
(\xi \cdot \bar{\hat{f}}_{1})}{|\xi|^{5}}
\in L^{1}(\mathbb{R}^{n};\mathbb{R})
\end{split}
\end{equation*}
for $n \ge 3$.
Thus we see that
\begin{equation*} 
\begin{split}
\| u_{N}(t) \|_{2} & \ge \| e^{-\beta |\xi|^{2} } \hat{V}_{N,c}(t) \|_{2} 
+\| e^{-\beta |\xi|^{2} } \hat{V}_{N,s}(t) ) \|_{2} +o(1) \\
& 
\ge \| e^{-\beta |\xi|^{2} } \hat{U}_{N,c}(t) \|_{2}
-\| e^{-\beta |\xi|^{2} } \hat{R}_{N,c}(t) \|_{2} +o(1)
\end{split}
\end{equation*}
as $t \to \infty$ by \eqref{eq:2.1} and \eqref{eq:2.2}.
Taking $\beta \ge \beta_{1}$, 
we apply the argument of the estimates \eqref{eq:4.12}-\eqref{eq:4.14} to have  
\begin{equation*}
\begin{split}
\| u_{N}(t) \|_{2} & \ge 
\frac{|M_{0}|}{2} 
\left(
\frac{
|\mathbb{S}^{n-1}| }{n(n+2)} 
2^{-\frac{n}{2}-1} \Gamma\left( \frac{n}{2} \right) \right)^{\frac{1}{2}} \beta^{-\frac{n}{4}} 
+o(1) 
\end{split}
\end{equation*}
as $t \to \infty$, which implies the estimate \eqref{eq:1.15} with $|M_{0}| \neq 0$ and $|M_{1}| = 0$.
We complete the proof of Theorem \ref{thm:1.1}.
%
\section{Proof of Theorem \ref{thm:1.4}}
This section is devoted to the proof of Theorem \ref{thm:1.4}.
As we make the assumption $f_{0} \in \{ L^{1,1} \}^{n}$ and $f_{1} \in \{ L^{1,1} \}^{n}$
in this section, we note the fact that
\begin{equation} \label{eq:5.1}
\begin{split}
\lim_{\beta \to \infty}
\sum_{k=1}^{n} \int_{|y| \ge \beta^{\frac{1}{4}}} |y| |f_{jk}(y)|dy=0
\end{split}
\end{equation}
for $j=0,1$.

We begin with estimates which are valid for all space dimension $n \ge 2$.
The direct calculation yields 
\begin{equation*}
\begin{split}
& \| e^{-\beta |\xi|^{2} } ( \hat{U}_{j,c}(t)+\hat{U}_{j,s}(t) + \hat{\tilde{U}}_{j,c}(t)+\hat{\tilde{U}}_{j,s}(t) ) \|_{2}^{2} \\
& = 
\| e^{-\beta |\xi|^{2} } \hat{U}_{j,c}(t)\|_{2}^{2} 
+\| e^{-\beta |\xi|^{2} } \hat{U}_{j,s}(t)  \|_{2}^{2} 
+
\| e^{-\beta |\xi|^{2} } \hat{\tilde{U}}_{j,c}(t) \|_{2}^{2}
+
\| e^{-\beta |\xi|^{2} } \hat{\tilde{U}}_{j,s}(t) \|_{2}^{2} +o(1)
\end{split}
\end{equation*}
by $|z|^{2}=x_{1}^{2}+x_{2}^{2}$ for $z=x_{1}+ix_{2} \in \mathbb{C}$ with $x_{1}, x_{2} \in \mathbb{R}$
and \eqref{eq:2.11}.
Thus noting that $A^{2}+B^{2} \ge \frac{1}{2}(A+B)^{2}$ for $A, B \ge 0$, 
we see that
\begin{equation} \label{eq:5.2}
\begin{split}
& \| u_{N}(t) \|_{2} \ge \| e^{-\beta |\xi|^{2} } \hat{u}_{N}(t) \|_{2} \\
& \ge 
C (\| e^{-\beta |\xi|^{2} } \hat{U}_{j,c}(t)\|_{2}^{2} 
+\| e^{-\beta |\xi|^{2} } \hat{U}_{j,s}(t)  \|_{2}^{2} 
+
\| e^{-\beta |\xi|^{2} } \hat{\tilde{U}}_{j,c}(t) \|_{2}
+
\| e^{-\beta |\xi|^{2} } \hat{\tilde{U}}_{j,s}(t) \|_{2})+o(1)\\
& 
-\| e^{-\beta |\xi|^{2} } \hat{\tilde{R}}_{j,c}(t) \|_{2}
-\| e^{-\beta |\xi|^{2} } \hat{\tilde{R}}_{j,s}(t) \|_{2}
\end{split}
\end{equation}
as $t \to \infty$.
In the proof of Theorem \ref{thm:1.4}, 
it is useful to express the terms $\| e^{-\beta |\xi|^{2} } \hat{\tilde{U}}_{j,c}(t) \|_{2}$ and 
$\| e^{-\beta |\xi|^{2} } \hat{\tilde{U}}_{j,s}(t) \|_{2}$ in a form that makes the $\beta$ dependence clear:
\begin{equation} \label{eq:5.3}
\begin{split}
\| e^{-\beta |\xi|^{2} } \hat{\tilde{U}}_{N,c}(t) \|_{2}^{2} = \frac{1}{2}
\int_{0}^{\infty} e^{-2 \beta r^{2} } r^{n+1} dr  (K_{01}^{(N)}-2K_{02}^{(N)} +2 K_{03}^{(N)})+o(1)
\end{split}
\end{equation}
and 
\begin{equation} \label{eq:5.4}
\begin{split}
& \| e^{-\beta |\xi|^{2} } \hat{\tilde{U}}_{N,s}(t) \|_{2}^{2} \\
& = \frac{1}{2}
\int_{0}^{\infty} e^{-2 \beta r^{2} } r^{n-1} dr 
\left\{  
\frac{1}{2 \alpha_{2}^{2}}
(K_{11}^{(N)}-2K_{12}^{(N)} +K_{13}^{(N)})+
\frac{1}{2 \alpha_{1}^{2}} K_{13}^{(N)}
\right\}
+
o(1)
\end{split}
\end{equation}
as $t \to \infty$, 
where 
\begin{equation} \label{eq:5.5}
\begin{split}
K_{j1}^{(N)} & := 
\int_{\mathbb{S}^{d-1}} 
(\omega \cdot P_{jN})^{2} d \omega,\\
K_{j2}^{(N)} & :=\sum_{k=1}^{n} \int_{\mathbb{S}^{n-1}}
\omega_{N} \omega_{k} 
(\omega \cdot P_{jN})(\omega \cdot P_{jk}) d \omega, \\
K_{j3}^{(N)} & :=  \int_{\mathbb{S}^{n-1}}
\omega_{N}^{2} 
\left( \sum_{k=1}^{n} \omega_{k} (\omega \cdot P_{jk}) \right)^{2} d \omega
\end{split}
\end{equation}
for $j=0,1$.
We note that \eqref{eq:5.3} and \eqref{eq:5.4} are valid even if we take $\beta=t>0$.

Now we collect the values of the spherical integrals defined by \eqref{eq:5.5} precisely.
%
\begin{lem} \label{Lem:5.1} 
Let $j= 0,1$.
The following statements hold.
\begin{equation} \label{eq:5.6}
\begin{split}
K_{j1}^{(N)} = \frac{|\mathbb{S}^{n-1}|}{n}|P_{j N}|^{2} 
=\frac{|\mathbb{S}^{n-1}|}{n} \sum_{k=1}^{n} p_{j N, k}^{2},
\end{split}
\end{equation}
\begin{equation} \label{eq:5.7}
\begin{split}
K_{j2}^{(N)} = \dfrac{|\mathbb{S}^{n-1}|}{n(n+2)} 
\displaystyle\sum_{ \substack{ k=1 \\ k \neq N} }^{n} 
(p_{jN,N} p_{j k, k} +p_{j N, k} p_{j k, N}+ p_{jN, k}^{2}) 
+\dfrac{3 |\mathbb{S}^{n-1}|}{n(n+2)} p_{jN, N}^{2},
\end{split}
\end{equation}
\begin{equation} \label{eq:5.8}
\begin{split}
K_{j3}^{(N)} & =\frac{ |\mathbb{S}^{n-1}|}{n(n+2)(n+4)}
\left\{ 
15 p_{jN,N}^{2} +  \sum_{ \substack{k=1 \\ k \neq N} }^{n} (3p_{jN,k}^{2}
+3p_{jk,N}^{2} 
+ 2 p_{jk,k}^{2} )
+ \left( 
\sum_{ \substack{ k=1 \\ k \neq N} }^{n} p_{jk,k} 
\right)^{2} 
\right. \\
& + 
\left. 
6 \sum_{ \substack{ k=1 \\ k \neq N} }^{n}  (p_{jN,N} p_{jk,k} + p_{jN,k} p_{jk, N} )
+\sum_{ \substack{ k, \ell=1 \\ k < \ell \\ k \neq N, \ell \neq N } }^{n} 
(p_{jk, \ell} + p_{j \ell, k} )^{2} 
\right\}.
\end{split}
\end{equation}
for $n \ge 3$ and 
\begin{equation} \label{eq:5.9}
\begin{split}
K_{j3}^{(N)} & =\frac{\pi}{8}
\left\{ 
5 p_{jN,N}^{2} + p_{j1,2}^{2} +p_{j2,1}^{2} +   
+2(p_{j1,1} p_{j2,2} + p_{j1,2} p_{j2, 1} )+
\sum_{ \substack{ k=1 \\ k \neq N} }^{2}  p_{jk,k}^{2}
\right\}.
\end{split}
\end{equation}
for $n=2$
\end{lem}
\begin{proof}
We easily have \eqref{eq:5.6} as a direct consequence of \eqref{eq:2.4} and \eqref{eq:2.6}.
In order to prove \eqref{eq:5.7}, 
we use \eqref{eq:2.4} and \eqref{eq:2.7} to see that
\begin{equation*}
\begin{split}
K_{j2}^{(N)} & =\frac{|\mathbb{S}^{n-1}|}{n(n+2)} 
\sum_{ \substack{ k=1 \\ k\neq  N} }^{n} (p_{jN,N} p_{jk, k} +p_{j N, k} p_{jk, N}) 
+ \int_{\mathbb{S}^{n-1}} \omega_{N}^{2} 
(\omega \cdot P_{jN})^{2} d \omega 
\end{split}
\end{equation*}
and 
\begin{equation*}
\begin{split}
\int_{\mathbb{S}^{n-1}} \omega_{N}^{2} 
(\omega \cdot P_{jN})^{2} d \omega 
& = \frac{2 |\mathbb{S}^{n-1}|}{n(n+2)} 
\sum_{ \substack{k=1 \\ k \neq N} }^{n} 
 p_{jN, k}^{2}+
\frac{3 |\mathbb{S}^{n-1}|}{n(n+2)} p_{jN, N}^{2},
\end{split}
\end{equation*}
which implies \eqref{eq:5.7}.
Next we prove \eqref{eq:5.8}.
To this end we decompose $K_{j3}^{(N)}$ into three parts:
\begin{equation} \label{eq:5.10}
\begin{split}
K_{j3}^{(N)} =K_{j3,1}^{(N)} +K_{j3,2}^{(N)} +K_{j3,3}^{(N)},
\end{split}
\end{equation}
where 
\begin{equation*}
\begin{split}
K_{j3,1}^{(N)}  & := \sum_{k=1}^{n} \int_{\mathbb{S}^{n-1}} \omega_{N}^{2} \omega_{k}^{2} 
(\omega \cdot P_{jk} )^{2} d \omega, \\
K_{j3,2}^{(N)}  & := 2 \sum_{ \substack{ k=1 \\ k \neq N} }^{n} 
\int_{\mathbb{S}^{n-1}} \omega_{N}^{3} \omega_{k} 
(\omega \cdot P_{jN} )(\omega \cdot P_{jk} ) d \omega, \\
K_{j3,3}^{(N)}  & := 2 \sum_{ \substack{ k, \ell=1 \\ k < \ell \\ k \neq N, \ell \neq N } }^{n}  
\int_{\mathbb{S}^{n-1}} \omega_{N}^{2} \omega_{k} \omega_{\ell}
(\omega \cdot P_{jk} )(\omega \cdot P_{j\ell} ) d \omega.
\end{split}
\end{equation*}
We apply \eqref{eq:2.4} and \eqref{eq:2.8} to have 
\begin{equation*}
\begin{split}
\int_{\mathbb{S}^{d-1}} \omega_{N}^{4} 
(\omega \cdot P_{jN} )^{2} d \omega
& = p_{jN,N}^{2} \int_{\mathbb{S}^{n-1}} \omega_{N}^{6} d \omega + 
\sum_{ \substack{ k=1 \\ k \neq N} }^{n} (p_{jN,k}^{2}+ p_{jk,N}^{2} ) 
\int_{\mathbb{S}^{n-1}} \omega_{N}^{4} \omega_{k}^{2} d \omega 
\end{split}
\end{equation*}
and 
\begin{equation*}
\begin{split}
& \sum_{ \substack{ k=1 \\ k \neq N} }^{n} 
\int_{\mathbb{S}^{n-1}} \omega_{N}^{2} \omega_{k}^{2} 
(\omega \cdot P_{jk} )^{2} d \omega \\
& =  \sum_{ \substack{ k=1 \\ k \neq N} }^{n} 
p_{jk,k}^{2}  \int_{\mathbb{S}^{n-1}} \omega_{N}^{2} \omega_{k}^{4} d \omega  
+  \sum_{ \substack{ k=1 \\ k \neq N} }^{n} 
\sum_{ \substack{ \ell=1 \\ \ell \neq N \\ \ell \neq k} }^{n} p_{jk, \ell}^{2} 
\int_{\mathbb{S}^{n-1}} \omega_{N}^{2} \omega_{k}^{2} \omega_{\ell}^{2} d \omega.
\end{split}
\end{equation*}
This leads 
\begin{equation} \label{eq:5.11}
\begin{split}
K_{j3,1}^{(N)}   
& = \frac{2 |\mathbb{S}^{n-1}|}{n(n+2)(n+4)}
\left\{ 
15 p_{jN,N}^{2} + 3 \sum_{ \substack{ k=1 \\ k \neq N} }^{n} 
(p_{jN,k}^{2} + p_{jk,N}^{2} + p_{jk,k}^{2}) 
+ \sum_{ \substack{ k=1 \\ k \neq N} }^{n} 
\sum_{ \substack{ \ell=1 \\ \ell \neq N, \ell \neq k} }^{n} p_{jk, \ell}^{2}
\right\}.
\end{split}
\end{equation}
By the similar way, we also have 
\begin{equation} \label{eq:5.12}
\begin{split}
K_{j3, 2}^{(N)} 
=\frac{6 |\mathbb{S}^{n-1}|}{n(n+2)(n+4)} \sum_{ \substack{ k=1 \\ k \neq N} }^{n} (p_{jN, N} p_{jk,k} +p_{jN,k} p_{jk,N})
\end{split}
\end{equation}
and 
\begin{equation} \label{eq:5.13}
\begin{split}
K_{j3,3}^{(N)} =\frac{|\mathbb{S}^{n-1}|}{n(n+2)(n+4)} 
\sum_{ \substack{ k, \ell=1 \\ k < \ell \\ k \neq N, \ell \neq N } }^{n} 
 (p_{jk,k} p_{j \ell, \ell} +p_{jk, \ell} p_{j \ell, k}).
\end{split}
\end{equation}
Observing that
\begin{equation*}
\begin{split}
\sum_{ \substack{ k=1 \\ k \neq N} }^{n}  p_{jk,k}^{2}
+ 
\sum_{ \substack{ k, \ell=1 \\ k < \ell \\ k \neq N, \ell \neq N } }^{n} 
p_{jk,k} p_{j\ell, \ell} = \left( \sum_{ \substack{ k=1 \\ k \neq N} }^{n}  p_{jk,k} \right)^{2},
\end{split}
\end{equation*}
we arrive at the expression of $K_{3}$, \eqref{eq:5.8} by \eqref{eq:5.10}-\eqref{eq:5.13}.
We can apply same argument to have \eqref{eq:5.9}.
We complete the proof of Lemma \ref{Lem:5.1}. 
\end{proof}

The terms
$\| e^{-\beta |\xi|^{2} } \hat{\tilde{R}}_{N,c}(t) \|_{2}$ and 
$\| e^{-\beta |\xi|^{2} } \hat{\tilde{R}}_{N,s}(t) \|_{2}$ in \eqref{eq:5.2} are estimated as follows:
Applying \eqref{eq:3.10} with $k=1$,
we see that
\begin{equation} \label{eq:5.14}
\begin{split}
\| e^{-\beta |\xi|^{2} } \hat{\tilde{R}}_{N,c}(t) \|_{2} & \le 
\left\| W^{(\alpha_{2}, \beta)}_{c}(t) \ast f_{0N}-\sum_{  |\gamma| \le 1} m_{f_{0N}, \gamma} 
\partial_{x}^{\gamma} W^{(\alpha, \beta)}_{c} (t) \right\|_{2} \\
& +\sum_{\ell=1}^{2} \sum_{k=1}^{n} 
\left\| \tilde{W}^{(\alpha_{\ell}, \beta)}_{c,(N,k)}(t) \ast f_{0k}
-\sum_{  |\gamma| \le 1} m_{f_{0k}, \gamma}  \tilde{W}^{(\alpha_{\ell}, \beta)}_{c, (N,k)}(t) \right\|_{2} \\
& \le C \beta^{-\frac{n}{4}-\frac{3}{4}} \| f_{0} \|_{L^{1,1}}
+C \beta^{-\frac{n}{4}-\frac{1}{2}} \sum_{k=1}^{n} \int_{|y| \ge \beta^{\frac{1}{4}}} 
|y||f_{0k}(y)|dy +o(1)
\end{split}
\end{equation}
as $t \to \infty$ for $\beta >0$.
Similarly, we have
\begin{equation} \label{eq:5.15}
\begin{split}
& \| e^{-\beta |\xi|^{2} } \hat{\tilde{R}}_{N,s}(t) \|_{2} 
\le C \beta^{-\frac{n}{4}-\frac{1}{4}} \| f_{1} \|_{L^{1,1}}
+C \beta^{-\frac{n}{4}} \sum_{k=1}^{n} \int_{|y| \ge \beta^{\frac{1}{4}}} 
|y||f_{1k}(y)|dy+o(1)
\end{split}
\end{equation}
as $t \to \infty$ for $\beta>0$ by \eqref{eq:3.11} with $k=1$.
\subsection{$n=2$}
In this subsection, we shall show the estimates \eqref{eq:1.18} and \eqref{eq:1.19}
First, we prove the estimate \eqref{eq:1.18}.
Employing the method in \cite{T1}, we easily have 
\begin{equation} \label{eq:5.16}
\begin{split}
\| \hat{U}_{N,s}(t) \|_{2} 
& \le C (|M_{1}|+|m_{1N}| ) \sqrt{\log(t+2)} \le C |M_{1}| \sqrt{\log(t+2)}
\end{split}
\end{equation}
for $t \ge 0$.
Applying the the mean value theorem 
\begin{equation*}
\begin{split}
\hat{f}_{0k} (\xi)-m_{0k}= \xi \cdot \nabla_{\xi} \hat{f}_{0k}(\theta_{k} \xi) 
\end{split}
\end{equation*}
for some $\theta_{k} \in (0,1)$ for $k=1, 2$, 
we see that 
\begin{equation} \label{eq:5.17}
\begin{split}
 \| \hat{R}_{N,s}(t)) \|_{2} \le C \| f_{1} \|_{L^{1,1}}, 
\end{split}
\end{equation}
where we used the fact that 
$\| \nabla \hat{f}_{0k} \|_{\infty} \le C \| f _{0k} \|_{L^{1,1}}$.
Therefore we obtain 
\begin{equation*} 
\begin{split}
\| u_{N}(t) \|_{2} & \le 
\| \hat{U}_{N,s}(t) \|_{2} + \| \hat{R}_{N,s}(t)) \|_{2}
+\|  \hat{U}_{N,c}(t)+\hat{R}_{N,c}(t) \|_{2} \\
& \le  C |M_{1}| \sqrt{\log(t+2)} + C \| f_{1} \|_{L^{1,1}} +C \| f_{0} \|_{2},
\end{split}
\end{equation*}
which is the desired estimate \eqref{eq:1.18}.
Next we consider the lower bound estimate \eqref{eq:1.19}.
It follows from \eqref{eq:5.3}, \eqref{eq:5.4}, \eqref{eq:5.6}, \eqref{eq:5.7}, \eqref{eq:5.9} and \eqref{eq:3.3} that
\begin{equation*} 
\begin{split}
\| e^{-\beta |\xi|^{2} } \hat{\tilde{U}}_{N,c}(t) \|_{2}^{2} 
& = \ 2^{-3} \beta^{-2}  (K_{01}^{(N)}-2K_{02}^{(N)} +2 K_{03}^{(N)})+o(1) \\
& \ge C \beta^{-2} (p_{01,1}^{2} +p_{01,2}^{2} +p_{02,1}^{2}+p_{02,2}^{2})
\end{split}
\end{equation*}
and 
\begin{equation*} 
\begin{split}
\| e^{-\beta |\xi|^{2} } \hat{\tilde{U}}_{N,s}(t) \|_{2}^{2} & = 2^{-2} \beta^{-1}
\left\{  
\frac{1}{2 \alpha_{2}^{2}}
(K_{11}^{(N)}-2K_{12}^{(N)} +K_{13}^{(N)})+
\frac{1}{2 \alpha_{1}^{2}} K_{13}^{(N)}
\right\}
+
o(1) \\
& \ge C\beta^{-1}
\left\{  
\frac{1}{\alpha_{2}^{2}}
(p_{11,2}^{2}+p_{12,1}^{2})+
\frac{1}{\alpha_{1}^{2}}(p_{11,1}^{2}+p_{12,2}^{2}) 
\right\}
+
o(1)
\end{split}
\end{equation*}
as $t \to \infty$.
Namely we obtain 
\begin{equation} \label{eq:5.18}
\begin{split}
\| e^{-\beta |\xi|^{2} } \hat{\tilde{U}}_{N,c}(t) \|_{2}
& \ge C_{1} \beta^{-1} (|P_{01}|+|P_{02}|) 
\end{split}
\end{equation}
and 
\begin{equation} \label{eq:5.19}
\begin{split}
\| e^{-\beta |\xi|^{2} } \hat{\tilde{U}}_{N,s}(t) \|_{2} & \ge C_{2} \beta^{-\frac{1}{2}}
(|P_{11}|+|P_{12}|)
\end{split}
\end{equation}
for some $C_{1}>0$ and $C_{2}>0$,
when $t \gg 1$.
We choose $\beta_{2}>0$ sufficiently large such that 
for all $\beta \ge \beta_{2}$, it holds that 
\begin{equation} \label{eq:5.20}
\begin{split}
& \| e^{-\beta |\xi|^{2} } \hat{\tilde{R}}_{N,c}(t) \|_{2} \le \frac{C_{1}}{2} \beta^{-1} (|P_{01}|+|P_{02}|) 
\end{split}
\end{equation}
by \eqref{eq:5.14} with $n=2$ and 
\begin{equation} \label{eq:5.21}
\begin{split}
& \| e^{-\beta |\xi|^{2} } \hat{\tilde{R}}_{N,s}(t) \|_{2} \le \frac{C_{2}}{2} \beta^{-\frac{1}{2}}
(|P_{11}|+|P_{12}|)
\end{split}
\end{equation}
by \eqref{eq:5.15} with $n=2$.
Therefore,
Using \eqref{eq:4.7}, \eqref{eq:4.12}, \eqref{eq:5.2} and \eqref{eq:5.18}-\eqref{eq:5.21}, 
we have 
\begin{equation*} 
\begin{split}
\| u_{N}(t) \|_{2} 
& \ge 
C |M_{0}| \beta^{-\frac{1}{2}} + C|M_{1}| \sqrt{\log(t+2)} \\
& +\frac{C_{1}}{2} \beta^{-1} (|P_{01}|+|P_{02}|) 
+\frac{C_{2}}{2} \beta^{-\frac{1}{2}}
(|P_{11}|+|P_{12}|)+o(1)
\end{split}
\end{equation*}
as $t \to \infty$, 
which is the desired estimate \eqref{eq:1.19}.
\subsection{$n=3$}
In this subsection, we shall give the proof of estimate \eqref{eq:1.20}.
Applying \eqref{eq:5.1}-\eqref{eq:5.3} with the estimates
\begin{equation*} 
\begin{split}
2(n-2)  p_{0N,N} \sum_{ \substack{ k=1 \\ k \neq N} }^{n} p_{0k,k} 
\le (n-2)^{2} p_{0N,N}^{2} 
+ \left( \sum_{ \substack{ k=1 \\ k \neq N} }^{n} p_{0k,k}  \right)^{2}
\end{split}
\end{equation*}
and 
\begin{equation*} 
\begin{split}
2(n-2) p_{0N,k} p_{0k, N} \le (n-2)^{2} p_{0N,k}^{2} + p_{0k, N}^{2},
\end{split}
\end{equation*}
we see that
\begin{equation} \label{eq:5.22}
\begin{split}
& K_{01}^{(N)}-2K_{02}^{(N)} +2 K_{03}^{(N)} \\
& \ge C
\left\{ 
(4n+10) p_{0N,N}^{2} +\left( 
\sum_{ \substack{ k=1 \\ k \neq N} }^{n} p_{0k,k} 
\right)^{2} 
+
(8n+2) \sum_{ \substack{ k=1 \\ k \neq N} }^{n} p_{0N,k}^{2}
+5 \sum_{ \substack{ k=1 \\ k \neq N} }^{n} p_{0k,N}^{2}
\right. \\
& +
\left. 4
\sum_{ \substack{ k=1 \\ k \neq N} }^{n}  p_{0k,k}^{2} 
+ \sum_{ \substack{ k, \ell =1 \\ k< \ell \\
k \neq N, \ell \neq N} }^{n}  (p_{0k, \ell} + p_{0 \ell , k} )^{2} 
\right\}.
\end{split}
\end{equation}
Then we apply \eqref{eq:5.3} and \eqref{eq:5.22} to obtain
\begin{equation} \label{eq:5.23}
\begin{split}
& \| e^{-\beta |\xi|^{2} } \hat{\tilde{U}}_{N,c}(t) \|_{2} \ge C \beta^{-\frac{n}{4}-\frac{1}{2}} |\mathbb{P}_{0}| 
\end{split}
\end{equation}
for some $C_{3}>0$.
Similar argument apply to the case $\| e^{-\beta |\xi|^{2} } \hat{\tilde{U}}_{N,c}(t) \|_{2}$. 
Namely, noting that
\begin{equation*}
\begin{split}
& K_{11}^{(N)}-2K_{12}^{(N)} +K_{13}^{(N)} \\
& =C \left\{ 
(n^{2}-1) \left( 
p_{1N,N}-\frac{1}{n-1} \sum_{ \substack{k=1 \\ k \neq N} }^{n} p_{1k,k}
\right)^{2} \right. \\
& +\sum_{ \substack{k=1 \\ k \neq N} }^{n} 
\left( \frac{2(n-3)}{n-1}  p_{1k,k}^{2}
+ (2n+1) p_{1N,k}^{2}
+ 2 
p_{1k,N}^{2} 
\right) \\
&  
\left. 
+ \sum_{ \substack{ k=1 \\ k \neq N} }^{n}  ( (n+1) p_{1N,k} +p_{1k, N} )^{2}
+\sum_{ \substack{ k, \ell=1 \\ k < \ell \\ k \neq N, \ell \neq N } }^{n} 
\left(
(p_{1k, \ell} + p_{1 \ell, k} )^{2} 
+ \frac{2}{n-1} 
(p_{1k,k}-p_{1 \ell,\ell})^{2}
\right)
\right\} 
\end{split}
\end{equation*}
and 
\begin{equation*}
\begin{split}
K_{13}^{(N)} & =C
\left\{ 
15 p_{1N,N}^{2} + 3 \sum_{ \substack{k=1 \\ k \neq N} }^{n} (p_{1N,k}
+p_{1k,N}^{2})^{2} 
+
2 \sum_{ \substack{k=1 \\ k \neq N} }^{n} p_{1k,k}^{2} 
+ \left( 
\sum_{ \substack{ k=1 \\ k \neq N} }^{n} p_{1k,k} 
\right)^{2} 
\right. \\
& + 
\left. 
6  p_{jN,N}\sum_{ \substack{ k=1 \\ k \neq N} }^{n}  p_{1k,k} 
+\sum_{ \substack{ k, \ell=1 \\ k < \ell \\ k \neq N, \ell \neq N } }^{n} 
(p_{1k, \ell} + p_{1 \ell, k} )^{2} 
\right\} \\
& \ge
C 
\left\{ 
6 p_{1N,N}^{2} + 3 \sum_{ \substack{k=1 \\ k \neq N} }^{n} (p_{1N,k}
+p_{1k,N}^{2})^{2} 
+
2 \sum_{ \substack{k=1 \\ k \neq N} }^{n} p_{1k,k}^{2} 
+\sum_{ \substack{ k, \ell=1 \\ k < \ell \\ k \neq N, \ell \neq N } }^{n} 
(p_{1k, \ell} + p_{1 \ell, k} )^{2} 
\right\}, 
\end{split}
\end{equation*}
we have
\begin{equation} \label{eq:5.24}
\begin{split}
& \| e^{-\beta |\xi|^{2} } \hat{\tilde{U}}_{N,s}(t) \|_{2} \ge C_{4} \beta^{-\frac{n}{4}}
|\mathbb{P}_{1}| 
\end{split}
\end{equation}
for some $C_{4}>0$ by \eqref{eq:5.4} and \eqref{eq:5.3}.
Again, thanks to the estimates \eqref{eq:5.14} and \eqref{eq:5.15}, 
we choose $\beta_{3}>0$ such that for all $\beta \ge \beta_{3}$, 
the following estimates hold:   
\begin{equation} \label{eq:5.25}
\begin{split}
& \| e^{-\beta |\xi|^{2} } \hat{\tilde{R}}_{N,c}(t) \|_{2} \le \frac{C_{3}}{2} \beta^{-\frac{n}{4}-\frac{1}{2}}
|\mathbb{P}_{0}| 
\end{split}
\end{equation}
and
\begin{equation} \label{eq:5.26}
\begin{split}
& \| e^{-\beta |\xi|^{2} } \hat{\tilde{R}}_{N,s}(t) \|_{2} \le \frac{C_{4}}{2} \beta^{-\frac{n}{4}}
|\mathbb{P}_{1}|.
\end{split}
\end{equation}
Summing up \eqref{eq:4.12}, \eqref{eq:4.13}, \eqref{eq:5.2} and \eqref{eq:5.23}-\eqref{eq:5.26}, 
we have 
\begin{equation*} 
\begin{split}
\| u_{N}(t) \|_{2} 
& \ge 
C \beta^{-\frac{n}{4}}|M_{0}|+ C\beta^{-\frac{n-2}{4}}|M_{1}| 
+\frac{C_{3}}{2} \beta^{-\frac{n}{4}-\frac{1}{2}}
|\mathbb{P}_{0}|
+\frac{C_{4}}{2} \beta^{-\frac{n}{4}}
|\mathbb{P}_{1}|+o(1)
\end{split}
\end{equation*}
as $t \to \infty$, 
which is the desired estimate \eqref{eq:1.20}.
We complete the proof of Theorem \ref{thm:1.4}.
%
\section{Appendix: wave equation}
In this appendix, we mention the result for the Cauchy problem of the wave equation \eqref{eq:1.3}. 
For the simplicity, we introduce the notation.
For the initial data $w_{0}$ and $w_{1}$, we define    
\begin{equation*}
\begin{split}
m_{j} := \int_{\mathbb{R}^{n}} w_{j}(x) dx, \quad p_{j,k} := -\int_{\mathbb{R}^{n}} x_{k} w_{j}(x) dx
\end{split}
\end{equation*} 
and
\begin{equation*}
\begin{split}
P_{j} := \mathstrut^{T}(p_{j,1}, p_{j,2}, \cdots,  p_{j,n})
\end{split}
\end{equation*} 
for $j=0,1$ and $k=1,2, \cdots, n$.
The lower bound estimates of the solution to \eqref{eq:1.3} is formulated as follows:
\begin{thm} \label{thm:6.1}
{\rm (i)}\ Suppose that $(u_{0}, u_{1}) \in H^{1} \cap L^{1} \times L^{2} \cap L^{1}$.
If either $m_{0} \neq 0$ or $m_{1} \neq 0$, 
the following estimate holds:
\begin{equation*} 
\begin{split}
\| u(t) \|_{L^{2}} 
\ge 
\begin{cases}
&  C (|m_{0}| + |m_{1}|t^{\frac{1}{2}} ), \quad n=1, \\
& C (|m_{0}|  + |m_{1}| \sqrt{\log(t+2)} ), \quad n=2, \\
&  C (|m_{0}| + |m_{1}|), \quad n \ge 3
\end{cases}
\end{split}
\end{equation*}
for $t \gg 1$.\\
{\rm (ii)}\ 
Suppose that $(u_{0}, u_{1}) \in H^{1} \cap L^{1,1} \times L^{2} \cap L^{1,1}$.
If either $P_{0} \neq 0$ or $P_{1} \neq 0$, 
the following estimate holds:
\begin{equation*} 
\begin{split}
\| u(t) \|_{L^{2}} 
\ge 
\begin{cases}
&  C (|m_{0}| +|P_{0}|+ |m_{1}|t^{\frac{1}{2}} +|P_{1}|), \quad n=1, \\
&  C (|m_{0}| + |P_{0}|+ |m_{1}| \sqrt{\log(t+2)} +|P_{1}|), \quad n=2, \\
&  C (|m_{0}| + |P_{0}| +|m_{1}| +|P_{1}|), \quad n \ge 3
\end{cases}
\end{split}
\end{equation*}
for $t \gg 1$.
\end{thm}
Theorem \ref{thm:6.1} is shown by the Theorem \ref{thm:1.1} by Lemmas \ref{Lem:2.3},
\ref{Lem:3.1} and \ref{Lem:3.2} and Proposition \ref{prop:3.3}.
We omit the detail. 

Combining the estimates derived in \cite{I2} and \cite{CheI1},
we obtain sharp $L^{2}$ estimates to the solution of \eqref{eq:1.3} for 
$(u_{0}, u_{1}) \in H^{1} \cap L^{1} \times L^{2} \cap L^{1}$.
\begin{cor}
Suppose that $(u_{0}, u_{1}) \in H^{1} \cap L^{1} \times L^{2} \cap L^{1}$.
If either $m_{0} \neq 0$ or $m_{1} \neq 0$, 
the following estimates hold:
\begin{equation*} 
\begin{split}
C (|m_{0}| +  |m_{1}|t^{\frac{1}{2}} )\le \| u(t) \|_{L^{2}} \le C (\| w_{0} \|_{L^{2}} 
+ t^{\frac{1}{2}} \| w_{1} \|_{L^{1}} + \| w_{1} \|_{L^{2}}  ), \quad n=1,
\end{split}
\end{equation*}
\begin{equation*} 
\begin{split}
C (|m_{0}| +|m_{1}|\sqrt{\log(t+2)} ) \le \| u(t) \|_{L^{2}} \le C (\| w_{0} \|_{L^{2}} 
+ \sqrt{\log(t+2)} \| w_{1} \|_{L^{1}} + \| w_{1} \|_{L^{2}}  ), \quad n=2,
\end{split}
\end{equation*}
and
\begin{equation*} 
\begin{split}
C (|m_{0}| + |m_{1}|) \le \| u(t) \|_{L^{2}} \le C (\| w_{0} \|_{L^{2}} +\| w_{1} \|_{L^{1}} + \| w_{1} \|_{L^{2}}  ), \quad n \ge 3
\end{split}
\end{equation*}
for $t \gg 1$.
\end{cor}

\vspace*{5mm}
\noindent
\textbf{Acknowledgments. }
\smallskip
H. Takeda was partially supported by
JSPS KAKENHI Grant Numbers JP24K06822.

\end{document}